\documentclass[11pt]{article}
\usepackage{multicol}
\usepackage{graphics}
\usepackage{amsmath}
\usepackage{amsfonts}
\usepackage{amssymb}
\usepackage[english]{babel}
\usepackage{multirow}
\usepackage{subfigure}
\usepackage[thmmarks,amsmath]{ntheorem}
\usepackage{epsfig}
\usepackage{epstopdf}
\usepackage{graphicx,color}
\usepackage{enumerate}
\usepackage{multicol}
\usepackage{makecell}
\usepackage{lscape}
\usepackage{rotating}
\usepackage{bm}
\usepackage{color}

\usepackage[title,titletoc,toc]{appendix}
\usepackage[left, pagewise]{lineno}

\usepackage{amsmath}

\allowdisplaybreaks[4]

\numberwithin{equation}{section}
\newtheorem{theorem}{Theorem}[section]

\newtheorem{corollary}[theorem]{Corollary}

\newtheorem{lemma}[theorem]{Lemma}

\newtheorem{proposition}[theorem]{Proposition}
\newtheorem{remark}[theorem]{Remark}

\newenvironment{proof}[1][Proof]{\textbf{#1.} }{\ \rule{0.5em}{0.5em}}

\usepackage{geometry}
\begin{document}

\begin{center}
\smallskip
{\Large\bf  Analysis of a parabolic-hyperbolic hybrid population model}

\vspace{0.1in}

Qihua Huang$^\dag$, Minglong Wang$^\ddag$, Yixiang Wu$^\S$


\footnote{$^\dag$School of Mathematics and Statistics, Southwest University, Chongqing 400715, China. (qihua@swu.edu.cn). The research of this author is  supported by the National Natural Science Foundation of China (12271445).} \\

\footnote{$^\ddag$School of Mathematics and Statistics, Southwest University, Chongqing 400715, China (mlwangmath@163.com).}\\

\footnote{$^\S$Corresponding author. Department of Mathematics, Middle Tennessee State University, Murfreesboro, TN 37132, USA (yixiang.wu@mtsu.edu).}

\end{center}

\vspace{-0.6in}

\noindent\hrulefill
\vspace{-0.0in}

\noindent {\bf Abstract:}
This paper is concerned with the global dynamics of a hybrid parabolic-hyperbolic model describing populations with distinct dispersal and sedentary stages. We first establish the global well-posedness of solutions, prove a  comparison principle, and demonstrate the asymptotic smoothness of the solution semiflow. Through the spectral analysis of the linearized system, we derive and characterize the net reproductive rate $\mathcal{R}_{0}$. Furthermore,  an explicit relationship between $\mathcal{R}_{0}$ and the principal eigenvalue of the linearized system is analyzed. Under appropriate monotonicity assumptions, we show that $\mathcal{R}_{0}$ serves as a threshold parameter that completely determines the stability of steady states of the system.  More precisely, when $\mathcal{R}_{0}<1$, the trivial equilibrium is globally asymptotical stable, while when $\mathcal{R}_{0}>1$, the system is uniformly persistent and there is a positive equilibrium which is unique and globally asymptotical stable.

\medskip

\noindent {\bf Keywords:}
Hybrid model, integrated semigroup,  comparison principle, net reproductive rate,  global stability.

\medskip

\noindent {\bf AMS subject classifications:} 92D25, 35K57, 35L50, 35B40

\noindent\hrulefill

\section{Introduction}
Understanding how the interaction between life history traits, vital rates, and dispersal patterns influences population persistence is an essential research question in spatial ecology. Mathematical models, particularly reaction-diffusion equations \cite{Cantrell2003, Murray2002a, Murray2002b, Okubo2001, Shigesada1997} and  {first-order  hyperbolic } equations \cite{iannelli1995mathematical, inaba2017age,li2020age, Ruan2018, metz2014dynamics, Webb1985}, are essential tools for investigating these dynamics. It is important to note that classical reaction-diffusion models typically assume that growth and dispersal processes occur simultaneously in both time and space.   {In contrast, classical first-order hyperbolic equations used to model age- or size-structured populations}  generally assume a homogeneously mixed population, which neglect the spatial dispersal of individuals.
The above assumptions may not accurately reflect the life cycle of many organisms. Actually, species such as plants and certain aquatic organisms, like zebra mussels and corals, have distinct dispersal and stationary stages in their life cycles. For example, plants   {spread} to new habitats through seed dispersal, and once the seeds land, they grow into trees or grasses. The life cycle of zebra mussels   {includes} a dispersive larval stage and sedentary juvenile and adult stages. During breeding seasons, adult mussels release planktonic larvae that   {are transported by water currents} until they settle on the bottom of water bodies (dispersal stage), where they  attach to fixed substrates and mature into juveniles and adults (sedentary stage).

  {To address  the limitations mentioned above}, researchers have developed hybrid systems to investigate the spatiotemporal dynamics of such species. These hybrid systems   {typically} combine reaction-diffusion (or reaction-diffusion-advection) equations, which model the dispersal phase, with ordinary differential equations \cite{Lutscher2006, Pachepsky2005, Seo2011} (or difference equations \cite{Huang2017, Lewis2012}), which govern the growth phase. Some studies integrate spatial movement and age growth within a single equation of the form \( u_t + u_a - \mathcal{A}u \), where \( \mathcal{A} \) represents either a diffusion operator~\cite{ducrot2011travelling,ducrot2021integrated,magal2004eventual,walker2018some,webb2008population} or an integral operator~\cite{ducrot2024age1,ducrot2024age2,kang2021nonlinear}.

In this work, we study the following parabolic-hyperbolic hybrid system, which models the spatiotemporal dynamics of a population with distinct dispersal and stationary stages:
\begin{equation}\label{model}
  \left\{
   \begin{array}{lll}
\partial_{t} u=d\Delta u +B(x, t)-(m(x)+e(x))u-c(x) u^{2}, \quad &x\in\Omega, t>0, \\
\partial_{t}w+\partial_{a}w=-\mu(x,a, P(t))w ,\quad &x\in\Omega, t>0,a\in (0,a_{\max}), \\
\partial_\nu u=0,~\quad &x\in \partial \Omega, t>0, \\
 w(x,0,t)=\chi(x,P(t))e(x)u(x,t),  \quad &x\in\Omega, t>0, \\
 u(x,0)=u_{0}(x), \quad &x\in\Omega, \\
 w(x,a,0)=w_{0}(x,a),\quad  &x\in\Omega, a\in (0, a_{\max}).
   \end{array}
  \right.
  \end{equation}
In this system, $u(x, t)$ represents  the spatial density of dispersing individuals at location $x\in\Omega$ and time $t$. The domain $\Omega$ denotes the population habitat, which is a bounded region in  $\mathbb{R}^{n}$  with a smooth boundary $\partial \Omega$.  The function $w(x, a, t)$ represents the density of sedentary individuals with age $a\in [0, a_{\max})$ at location $x$ and time $t$. The first equation of system \eqref{model} describes the movement, recruitment, mortality, and settlement of dispersing individuals.  The term $d\Delta u$ models the random movement, where  $d$ is the diffusion coefficient and $\Delta $ denotes the Laplacian operator. The function
\begin{equation}\label{B}
B(x, t):=\int_{0} ^{a_{\max}}\beta(x,a, P(t))w(x,a,t)da
\end{equation}
represents the total contributions of all stationary individuals to recruitment, where  $\beta(x,a, P)$ is the spatially varying reproduction rate of stationary individuals of age $a$. This rate depends on the total number of stationary individuals $P(t)$, which reflects   competition for resources. Specifically, $P(t)$ is given by
$$
P(t):=\int_0^{a_{\max}} \int_\Omega w(x, a, t)dxda, \quad\forall t\ge 0.
$$
 The parameter $m(x)$ is the natural mortality rate. The term $c(x) u^2$ denotes mortality due to intra-specific competition, where  $c(x)$ is the competition coefficient. The parameter $e(x)$ is the settlement rate of dispersing individuals. The second equation of system \eqref{model}  describes the growth and mortality of   stationary individuals, structured by age, where $\mu(x, a, P(t))$ is the mortality rate of  sedentary individuals.  The third equation imposes a homogeneous Neumann boundary condition, indicating that no individuals can enter or exit  the habitat through the boundary,  where  $\nu$ is the outward unit normal vector to $\partial \Omega$.
 The fourth equation describes the transition of  dispersing individuals to stationary individuals once they settle. Here, $\chi(x, P(t))\in (0, 1]$ represents the proportion of dispersing individuals that successfully transition to the sedentary stage. Finally,  $u_{0}(x)$ and $w_{0}(x,a)$ represent the initial spatial distributions of dispersing individuals and stationary individuals, respectively.

 The parabolic-hyperbolic population model \eqref{model} extends the models developed  by Deng and Huang \cite{Deng2019,Deng2021}. In these earlier works, the authors assumed that the transition rate
$\chi(x, P(t))\equiv 1$. Additionally, in \cite{Deng2019}, they did not consider the dependence of the reproduction rate $\beta$ and the mortality rate $\mu $ on the total population abundance $P$. In their works \cite{Deng2019,Deng2021}, Deng and Huang defined weak solutions (in integral form) for these models by introducing smooth test functions. They then established the existence and uniqueness of these weak solutions using the monotone method based on a comparison principle. Furthermore, they investigated the criteria for population persistence through four interrelated measures. Their study also included numerical simulations to explore the influence of population dispersal, reproduction, settlement, and habitat boundaries on population persistence.

The  objective of this work is to analyze the global dynamics of the hybrid system \eqref{model}, focusing on the existence, uniqueness, boundedness, and smoothness of solutions, the analysis of the net reproductive rate $R_0$, as well as the stability of equilibria.   { We highlight two major mathematical challenges in analyzing system \eqref{model}.
First, proving the compactness or asymptotic smoothness of the model's solutions is not straightforward, as the first-order hyperbolic equation in the model is spatially dependent.}
Second, the dependence of
$\beta$, $\mu$, and $\chi$ on the total number of sedentary individuals, $P(t)$, results in a strong coupling between the reaction-diffusion and hyperbolic equations. Therefore, this system cannot be transformed into a delay differential equation by solving the hyperbolic equation using the characteristic method.   { These challenges pose significant obstacles} in analyzing the dynamic behavior of the system \eqref{model}.

To analyze model \eqref{model}, we will adopt the integrated semigroup approach. Since $\chi$ depends on $P$, it is not possible to reformulate \eqref{model} as a densely defined abstract Cauchy problem, which is a classical method   { for studying} evolution equations \cite{Pazy1983}. Instead, \eqref{model} takes the form of a non-densely defined Cauchy problem, for which integrated semigroup theory provides a suitable analytical framework (see \cite{arendt1987resolvent,arendt1987vector,da1987differential,kellerman1989integrated,Ruan2018,neubrander1988integrated,thiemea1990integrated,xiao2013cauchy} and references therein). Alternatively, one could integrate the second equation in \eqref{model} along characteristic lines, reformulating the system into a coupled problem involving a reaction-diffusion equation and an integral equation. Nevertheless, we opt for the integrated semigroup method to leverage its mature theoretical foundation, such as well-established results on the existence, uniqueness, and positivity of solutions.

Importantly, the solution obtained via the integrated semigroup method also satisfies the system derived from integration along characteristic lines (see Proposition \ref{prop_int}), a fact we exploit to establish the asymptotic smoothness of solutions in Section 4. It seems that the general theory of eventual compactness for integrated semigroups developed by Magal and Thieme \cite{magal2004eventual} does not directly apply to our model. Our approach is   { inspired} by ideas from Webb \cite{Webb1985}, particularly the decomposition of solutions as in Lemma \ref{asymptotically smooth theory}. 

One goal of this work is to demonstrate the effectiveness of monotone methods in analyzing the asymptotic behavior of solutions to first-order hyperbolic equations. Monotone techniques, based  on comparison principles, have long played an essential role in the study of the dynamics of ordinary differential equations and reaction-diffusion systems (see, e.g., \cite{smith1995monotone}). More recently, Magal et al. \cite{Magal2019} extended the theory of monotone semiflows and established a comparison principle for abstract semi-linear Cauchy problems with non-dense domains. This principle has been adopted by Ducrot et al.  \cite{ducrot2024age2} to construct upper/lower solutions and study the global stability of positive equilibrium. 
In Section 2.4, we apply the abstract results of Magal et al. \cite{Magal2019} directly to our model. Our contribution in this direction is to show that the monotone solutions,  as constructed in Theorem \ref{Increasing and Decreasing Solution}, converge to an equilibrium under the condition that the $\omega$-limit set is compact (see Theorem \ref{Theorem_c}). Additional applications of the comparison principle and Theorem \ref{Theorem_c}   are presented in Section 6, where we analyze the global asymptotic behavior of the solutions in terms of the net reproductive rate $\mathcal{R}_0$.

  {The paper is organized as follows.}  In the next section, we prove the existence and uniqueness of global solutions using the integral semiflow approach. We  also establish the positivity of solutions and the comparison principle. In section 3, we investigate the boundedness and global existence properties of the solution. Section 4 provides a proof of the asymptotic smoothness of the solution semiflow. In section 5, through spectral analysis of the generator for the strongly continuous semigroup associated with the linearized system at the zero equilibrium, we introduce a biologically meaningful   {net} reproductive rate
 $\mathcal{R}_0$ and characterize its precise relationship with the principal eigenvalue of the linearized system. In section 6, we demonstrate the global stability of steady states, which is determined by the    {net} reproductive  rate $\mathcal{R}_0$: when   {$\mathcal{R}_0$} is below unity, the trivial equilibrium is globally asymptotically stable; when   {$\mathcal{R}_0$} exceeds unity, we prove the existence and global stability of a positive steady state. Finally, in section 7, we summarize our findings and propose directions for future research.

\section{Well-posedness}\label{Well-posedness}

In this section, we  establish the well-posedness of system \eqref{model} by applying the integrated semigroup theory (see  \cite{arendt1987resolvent,arendt1987vector,da1987differential,kellerman1989integrated,Ruan2018,neubrander1988integrated,thiemea1990integrated,xiao2013cauchy}).

Throughout the paper, let $Y:=C(\bar{\Omega},\mathbb{R})$ be  the space of continuous functions from $\bar{\Omega}$ to $\mathbb{R}$ with the  usual supremum norm $\|\cdot\|_{\infty}$, and denote  $Y_{+}:=C(\bar{\Omega},\mathbb{R}_{+})$. For $q\in Y$, we denote $\bar q=\max_{x\in\bar\Omega} q(x)$ and $\underline q=\min_{x\in\bar\Omega} q(x)$.
We impose the following assumptions:

\begin{itemize}
\item[(A1)] $e, m, c\in Y$, and $e(x), m(x), c(x)>0$ for all $x\in\bar\Omega$.
\item[(A2)] $\chi\in C(\mathbb{R}, Y_+)$, and for any $\zeta>0$ and $P_1, P_2\in [-\zeta, \zeta]$, there exists $K>0$ such that
$$
\|\chi(\cdot, P_1)-\chi(\cdot, P_2)\|_{\infty}\le K |P_1-P_2|.
$$
Moreover, $0<\chi(x, P)\le 1$ for all $x\in\bar\Omega$ and $P\ge 0$.
\item[(A3)] $\mu, \beta\in C([0, a_{\max})\times \mathbb{R}, Y_+)\cap L^\infty((0, a_{\max})\times \mathbb{R}, Y_+)$, and for any $\zeta>0$ and $P_1, P_2\in [-\zeta, \zeta]$, there exists $K>0$ such that
$$
|\mu(x, a, P_1)-\mu(x,a, P_2)|\le K |P_1-P_2|
$$
uniformly for $(x, a)\in\bar\Omega\times [0, a_{\max})$.
\end{itemize}

\subsection{Local existence of solutions}

In order to apply integral semigroup theory, we need to  rewrite system \eqref{model} as an abstract Cauchy problem.
Let $Z:=L^{1}((0,a_{\max}),Y)$ equipped with the norm
$$
\|w\|_{Z}=\int_{0}^{a_{\max}}\|w(a)\|_\infty da, \quad \forall w\in Z.
$$
 For convenience, we  write $w(x,a)$ as $w(a)$ to hide the spatial variable $x$ for $w\in Z$. 

Let $A_1$ be a linear operator on $Y$ defined by
$$
A_1 u=d\Delta u,\quad\forall u\in D(A_1)
$$
with domain
$$
D(A_1)=\{u\in \cap_{p\ge 1} W^{2,p}(\Omega): \Delta u\in Y \ \ \text{and}\ \   \partial_\nu u=0 ~\text{on}~\partial\Omega\}.
$$
It is well-known that $A_1$ generates a strongly continuous semigroup $\{T_{A_1}(t)\}_{t\ge 0}$ on $Y$.

 Define
$$
X=Y\times Z\quad\text{and}\quad X_{0}=\{0_Y\}\times Z.
$$
Fix $\mu_0\in L^\infty(\bar\Omega\times (0, a_{\max}))$ and denote $\underline\mu_0=\text{essinf} \ \mu_0$. Define a closed linear operator $A_2:D(A_2)\subset X\rightarrow X$  by
\begin{equation*}
A_2\begin{gathered}\begin{pmatrix}
0\\ w
\end{pmatrix}
=\begin{pmatrix}
-w(0)\\ -\partial_a w-\mu_0 w
\end{pmatrix}
\end{gathered}
\end{equation*}
with domain
 $$
 D(A_2)=\{0_{Y}\}\times W^{1,1}((0,a_{\max}),Y),
 $$
where $W^{1,1}((0,a_{\max}),Y)$ is the Sobolev space consisting of functions $w\in Z$ with norm
$$
\|w\|_{Z}+\|\partial_a w\|_Z<\infty.
$$
Clearly, the closure of $D(A_2)$ is
$$
\overline{D(A_2)}=\{0_{Y}\}\times Z=X_0 \subsetneqq X,
$$
so $D(A_2)$ is not dense in $X$.



The following lemma shows that $A_2$ is a Hille-Yosida operator.
\begin{lemma}\label{H-Y}
If $\lambda \in \mathbb{R}$ with $\lambda >-\underline\mu_0$, then  $\lambda \in \rho(A_2)$ (the resolvent set of $A_2$) and
\begin{equation}\label{re_estimate}
\|(\lambda I-A_2)^{-n}\|\leq \frac{1}{(\lambda+\underline{\mu}_0)^{n}},\quad \forall n\ge1.
\end{equation}
\end{lemma}
\begin{proof}
Suppose that  $\lambda\in\mathbb{R}$ with $\lambda >-\underline{\mu}_0$. For any $(y,\vartheta)\in X$ and $(0_Y, w)\in X$, it is easy to see that
\begin{equation}\label{resolvent formula}
\begin{gathered}
\begin{pmatrix}
y\\ \vartheta
\end{pmatrix}
=(\lambda I-A_2)
\begin{pmatrix}
0_Y\\ w
\end{pmatrix}
 \Longleftrightarrow  w(a)=e^{-\int_{0}^{a}(\mu_0(l)+\lambda) dl}y +\int_{0}^{a}e^{-\int_{s}^{a}(\mu_0(l)+\lambda) dl}\vartheta(s)ds.
\end{gathered}
\end{equation}
If $(y, \vartheta) =(0_Y, 0_Z) $, then $w=0_Z$ and so $ker (\lambda I-A_2)=\{(0_Y, 0_Z) \}$. Since $\lambda >-\underline{\mu}_0$, it is easy to verify that $(0_Y,w)\in D(A_2)$ for any $(y,\vartheta)\in X$, and so $ran (\lambda I-A_2)=X$. Therefore, the operator $\lambda I-A_{2}$ is invertible and $\lambda\in \rho(A_2)$.

By \eqref{resolvent formula}, for any  $(y,\vartheta)^{T}\in X$, we have
\begin{equation*}
\begin{aligned}
\left\|(\lambda I-A_2)^{-1}\begin{pmatrix}
y\\ \vartheta
\end{pmatrix}\right\|_{X} &\leq \int_{0}^{a_{\max}}\|e^{-\int_{0}^{a}(\mu_0(l)+\lambda) dl} y\|_\infty da +\int_{0}^{a_{\max}}\int_{0}^{a}\|e^{-\int_{s}^{a}(\mu_0(l)+\lambda) dl}\vartheta(s)\|_{\infty}dsda
\\
 &\leq \int_{0}^{a_{\max}}e^{-(\lambda+\underline{\mu}_0) a}da\|y\|_\infty+\int_{0}^{a_{\max}}\int_{0}^{a}e^{-(\lambda+\underline{\mu}_0) (a-s)}\|\vartheta(s)\|_{\infty}dsda
 \\
 & \leq \frac{1}{\lambda+\underline{\mu}_0}\|y\|_\infty+\frac{1}{\lambda+\underline{\mu}_0}\|\vartheta\|_{Z}
 \\
 & =\frac{1}{\lambda+\underline{\mu}_0}\left\|\begin{pmatrix}
y\\ \vartheta
\end{pmatrix}\right\|_{X}.
\end{aligned}
\end{equation*}
Thus, \eqref{re_estimate} holds.
\end{proof}


Let
$$
\mathbb{X}:=Y\times X,\quad \mathbb{X}_{0}:=Y\times \overline{D(A_2)}=Y\times X_{0},
$$
and define linear operator $A: D(A)\subset \mathbb{X}\rightarrow \mathbb{X}$  by
\begin{equation}\label{A}
A\begin{gathered}\begin{pmatrix}
u\\ \begin{pmatrix}
0_Y\\ w
\end{pmatrix}
\end{pmatrix}
=\begin{pmatrix}
A_1 u\\ A_2\begin{pmatrix}
0_Y\\w
\end{pmatrix}
\end{pmatrix}
=\begin{pmatrix}
d\Delta u \\ -w(0) \\-\partial_a w-\mu_0 w
\end{pmatrix}
\end{gathered},
\end{equation}
with
$$
D(A)=D(A_{1})\times D(A_2)\subset Y\times X\quad\text{and} \quad \overline{D(A)}=\mathbb{X}_0\subsetneqq \mathbb{X}.
$$

By Lemma \ref{H-Y} and the fact that $A_1$ generates a strongly continuous semigroup, $A$ is also a Hille-Yosida operator. Since  $\overline{D(A)}=\mathbb{X}_{0}$, we introduce $A_{0}$, the part of $A$ in $\mathbb{X}_{0}$:
$$
A_{0}=A~~\text{on}~~D(A_{0})=\{\bm x\in D(A): A\bm x\in \mathbb{X}_{0}\}=\{(u, (0_Y, w)) \in D(A):\ w(0)=0_Y\}.
$$
Then $A_{0}$ is densely defined in $\mathbb{X}_{0}$, and
it follows from \cite[Theorem 1.5.3]{Pazy1983} that the part $A_{0}$ of $A$ generates a strongly continuous semigroup $\{T_{A_0}(t)\}_{t\geq 0}$ on $\mathbb{X}_{0}$. It is easy to see that for $(u, (0_Y, w)) \in D(A_0)$,
$$
T_{A_0}(t)
\begin{pmatrix}
u\\ \begin{pmatrix}
0_Y\\w
\end{pmatrix}
\end{pmatrix} =
\begin{pmatrix}
T_{A_1}(t)u, \left(0_Y, \
\left\{
\begin{array}{lll}
e^{-\int_{a-t}^a\mu_0(l)dl}w(a-t), \ \ \ &\text{if}\ a\ge t\\
0_Y, \ \ \ &\text{if} \ a< t
\end{array}
\right.
\right)
\end{pmatrix} , \quad\forall t\ge 0.
$$
Clearly, $A_0=(A_1, A_{2_0})$, where $A_{2_0}$ is the part of $A_2$ on $X_0$. Moreover, $A_{2_0}$ generates a strongly continuous semigroup $\{T_{A_{2_0}}(t)\}_{t\ge 0}$ on $X_0$.


Since $A$ is a Hille-Yosida operator on $\mathbb{X}$, it generates an integrated semigroup $\{S_{A}(t)\}_{t\geq 0}$  on $\mathbb{X}$, defined by
 $$
 S_A(t)=(\lambda I-A_{0})\int_{0}^{t} T_{A_0}(s)ds(\lambda I-A)^{-1},\quad \forall \lambda\in \rho(A).
 $$
 Let $g\in C([0, \tau], \mathbb{X})$ with $\tau>0$. Define
 $$
(S_A*g)(t):= \int_0^t S_A(t-s)g(s)ds.
 $$
Then $S_A*g$ is continuously differentiable, and $(S_A*g)(t)\in D(A)$ for any $t\in [0, \tau]$. Introduce the notation
$$
(S_A\diamond g)(t):=\frac{d}{dt}(S_A*g)(t).
$$
Then $(S_A\diamond g)(t)\in \mathbb{X}_0$, and the following approximation formula holds
$$
(S_A\diamond g)(t)=\lim_{\lambda\to\infty} \int_0^t T_{A_0}(t-s) \lambda (\lambda I-A)^{-1}g(s) ds, \ \ \forall t\in [0, \tau].
$$
We refer to \cite{Ruan2018} for the above results on integrated semigroups.

Define the nonlinear map $F: \mathbb{X}_{0}\rightarrow \mathbb{X}$ by
\begin{equation*}
F\begin{gathered}\begin{pmatrix}
u\\ \begin{pmatrix}
0_Y\\ w
\end{pmatrix}
\end{pmatrix}
=\begin{pmatrix}
f(u,w)\\ \begin{pmatrix}
g_1(u,\tilde w)\\g_2(\cdot, w)
\end{pmatrix}
\end{pmatrix},
\end{gathered}
\end{equation*}
where
$$
f(u,w)=\int_{0} ^{a_{\max}}\beta(\cdot, a, \tilde w) w(a)da-(m+e)u-c u^{2},
$$
and
$$
g_1(u,\tilde w)=\chi(\cdot, \tilde w)eu, \ \ g_2(a, w)=\mu_0(\cdot, a) w-\mu(\cdot,a, \tilde w) w
$$
for any $u\in Y$ and $w\in Z$ with $\tilde w:=\int_0^{a_{\max}} \int_\Omega w(x,a) dxda$.


With the above terminology, we can transform \eqref{model} into  the following  non-densely
defined abstract Cauchy problem
\begin{equation}\label{Cauchy problem}
 \begin{aligned}
 \frac{d\bm x}{dt}=A\bm x(t)+F(\bm x(t)),\quad t\geq 0,
   \end{aligned}
  \end{equation}
with
\begin{equation*}
 \begin{aligned}
 \bm x(t)= \begin{pmatrix}
u(\cdot, t)\\ \begin{pmatrix}
0_Y\\ w(\cdot,\cdot, t)
\end{pmatrix}
\end{pmatrix}
\quad \text{and} \quad \bm x(0)=\bm x_0:
=\begin{pmatrix}
u_{0}\\ \begin{pmatrix}
0_Y\\w_{0}
\end{pmatrix}
\end{pmatrix}\in \mathbb{X}_{0}.
   \end{aligned}
  \end{equation*}

Under the assumptions (A1)-(A3),  the  function  $F$ is Lipschitz continuous on bounded sets of $\mathbb{X}_{0}$ in the sense that for each $\zeta>0$, there exists a positive constant $K$ such that
\begin{equation}\label{Lipschitz continuous}
 \begin{aligned}
\|F(\bm x_2)-F(\bm x_1)\|_\mathbb{X} \leq K \|\bm x_2-\bm x_1\|_{\mathbb{X}},
   \end{aligned}
 \end{equation}
for any $\bm x_1, \bm x_2\in \mathbb{X}_{0}$ with $\|\bm x_1\|_{\mathbb{X}}\leq \zeta$ and $\|\bm x_2\|_{\mathbb{X}}\leq \zeta$. Based on the above abstract semigroup formulation, by \cite[Theorem 5.2.7]{Ruan2018}, we have the following result.

\begin{theorem}\label{solution semiflow}
Suppose that (A1)-(A3) hold. Then there exists a uniquely determined  continuous semiflow $\{U(t)\}_{t\geq 0}$ on $\mathbb{X}_{0}$ such that  for each $\bm x_0\in \mathbb{X}_{0}$,  the Cauchy problem \eqref{Cauchy problem} has a unique integrated solution (or mild solution) $U(\cdot)\bm x_0\in C([0,t_{\max}),\mathbb{X}_{0})$. That is, $U(t)\bm x_0$ satisfies
$$
\int_{0}^{t}U(s)\bm x_0 ds\in D(A),\quad \forall t\in [0,t_{\max}),
$$
\text{and}
\begin{equation}\label{integrated solution}
 \begin{aligned}
U(t)\bm x_0=\bm x_0+A\int_{0}^{t}U(s)\bm x_0 ds+\int_{0}^{t}F(U(s)\bm x_0)ds, \quad \forall t\in [0,t_{\max}).
   \end{aligned}
  \end{equation}
   Here, either  $t_{\max}=\infty$ or $t_{\max}<\infty$ and $\|U(t)\bm x_0\|_\mathbb{X}\to \infty$ as $t\to t_{\max}$. Moreover, the solution $U(t)\bm x_0$ has the following variation of constant form:
\begin{eqnarray}\label{Upsi}
U(t)\bm x_0&=&S_A'(t)\bm x_0+ S_A\diamond F(U(\cdot+s)\bm x_0)(t-s)\\
&=& T_{A_0}(t)\bm x_0+\lim_{\lambda\to\infty}\int_0^t T_{A_0}(t-s)\lambda (\lambda I-A)^{-1} F(U(s)\bm x_0) ds, \ \ \forall t\in [0, t_{\max}). \nonumber
\end{eqnarray}
\end{theorem}

\subsection{Integration along characteristic lines}

In this subsection, we demonstrate that \eqref{Upsi} yields solutions in a form derived via integration along the characteristic lines. This representation of the solution will facilitate the analysis of the asymptotic smoothness of the semiflow induced by \eqref{model} later.

\begin{proposition}\label{prop_int}
Suppose that (A1)-(A3) hold. For any $\bm x_0\in \mathbb{X}_{0}$,
let $U(\cdot)\bm x_0\in  C([0,t_{\max}),\mathbb{X}_{0})$ be the unique integrated solution of  \eqref{model}. Then $u\in C([0, t_{\max}), Y)\cap C^1((0, t_{\max}), Y)\cap C((0, t_{\max}), D(A_1))$ and  satisfies
\begin{equation}\label{uequ}
\partial_{t} u=d\Delta u +\int_{0} ^{a_{\max}}\beta(x,a, P)\omega(x,a,t)da  -(m(x)+e(x))u-c(x) u^{2}, \quad x\in\Omega, t>0.
\end{equation}
And, $w$ satisfies
\begin{equation}\label{wint}
    w(x, a, t)=
      \begin{cases}
    w_0(x, a-t) e^{-\int_0^t \mu(x, a-s, P(t-s))ds}, \ &\text{if}\ a> t, \\
\chi(x, P(t-a))e(x) u(x, t-a) e^{-\int_0^a
\mu(x, a-s, P(t-s))ds}, \ &\text{if}\ a<t,
\end{cases}
\end{equation}
for any $x\in\bar\Omega$ and $t\in [0, t_{\max})$, where $P(t)=\int_0^{a_{\max}}\int_\Omega w(x, a, t)dxda$.
\end{proposition}
\begin{proof}
By Theorem \ref{solution semiflow}, \eqref{model} has a unique  integrated solution  $U(t)\bm x_0=(u(\cdot, t), (0_Y, w(\cdot, \cdot, t))) $, $t\in [0, t_{\max})$. Since $A_1$ is the generator of a strongly continuous semigroup on $Y$, the approximation formula
$$
\lim_{\lambda\to\infty}\lambda (\lambda-A_1)^{-1}v=v
$$
holds for any $v\in Y$. The $u$ component of \eqref{Upsi} is
$$
u(\cdot, t)=T_{A_1}(t)u_0+\int_0^t T_{A_1}(t-s)\lambda (\lambda I-A)^{-1}f(u(\cdot, s), w(\cdot, a, s))ds.
$$
Taking $\lambda\to\infty$,  we obtain
$$
u(\cdot, t)=T_{A_1}(t)u_0+\int_0^t T_{A_1}(t-s)f(u(\cdot, s), w(\cdot, a, s))ds.
$$
Since $w\in C([0, t_{\max}), Z)$, we have $P=\int_0^{a_{\max}} \int_\Omega w(x, a, \cdot)dxda \in C([0, t_{\max}), \mathbb{R})$. By the assumptions on $\beta$ in (A3), $B\in C([0, t_{\max}), Y)$, where $B$ is defined by \eqref{B}.  So by the regularity theory of parabolic equations, we know $u\in C([0, t_{\max}), Y)\cap C^1((0, t_{\max}), Y)\cap C((0, t_{\max}), D(A_1))$ and $u$ satisfies \eqref{uequ}.

Next, we turn to the equations of $w$. In the following computations, we write $T_{A_{2_0}}(t)(0_Y, w_0) $ as $T_{A_{2_0}}(t)w_0$ to save space.  By  \eqref{resolvent formula} and \eqref{Upsi}, $w$ satisfies
\begin{equation}\label{we1}
\begin{array}{lll}
    w(\cdot, a, t)&=&T_{A_{2_0}}(t)w_0+\lim_{\lambda\to\infty}\int_0^t T_{A_{2_0}}(t-s)\lambda (\lambda I-A_2)^{-1} G(U(s)\bm x_0) ds\\
    &=&
\begin{cases}
     \pi(a, a-t)w_0(\cdot, a-t) \ &\text{if}\ a> t \\
     0_Y \ &\text{if}\ a<t
\end{cases}
+L,
    \end{array}
\end{equation}
where $G(U(s)\bm x_0):=(g_1(u(\cdot, s), \tilde w(s)), g_2(\cdot, w(\cdot, \cdot, s)))$, $\pi(a_1, a_2):=e^{-\int_{a_2}^{a_1}\mu_0(l)dl}$ for any $0\le a_1\le a_2<\infty$ and
$$
L=\lim_{\lambda\to\infty}\lambda \int_0^t T_{A_{2_0}}(t-s) \left[\pi(a, 0)e^{-a\lambda}g_1(u(\cdot, s),  P(s))+ \int_0^a \pi(a, \tau)e^{-\lambda(a-\tau)}g_2(\tau, w(\cdot, \tau, s))d\tau\right] ds.
$$
Now we compute $L$. Since
\begin{equation*}
T_{A_{2_0}}(t-s)w'=
    \begin{cases}
  \pi(a, a-t+s) w'(a-t+s), \quad &\text{if}\ s>t-a, \\
  0_Y, \quad &\text{if}\ s<t-a,
    \end{cases}
\end{equation*}
for any $w'\in Z$ with $w'(0)=0_Y$,
 we have
\begin{eqnarray}\label{we3}
&&\lim_{\lambda\to\infty}\lambda \int_0^t T_{A_{2_0}}(t-s) \pi(a, 0)e^{-a\lambda}g_1(u(\cdot, s), P(s))ds\nonumber\\
&&=\lim_{\lambda\to\infty}\lambda\int_{\max\{t-a,0\}}^t
   \pi(a, a-t+s) \pi(a-t+s, 0)e^{-(a-t+s)\lambda}g_1(u(\cdot, s), P(s))ds\nonumber\\
&&=
\begin{cases}
     0_Y \ &\text{if}\ a> t, \\
    \pi(a, 0)g_1(u(\cdot, t-a), P(t-a))  \ &\text{if}\ a<t,
\end{cases}
\end{eqnarray}
where we have used the following observation:
$$
\lim_{\lambda\to \infty} \lambda\int_{c_1}^{c_2} e^{-\lambda (s-k)} h(s)ds=
  \begin{cases}
     0, \ &\text{if}\ k<c_1,\\
     h(c_1),\ &\text{if}\ k=c_1,
\end{cases}
$$
for any $h\in C([c_1, c_2], \mathbb{R})$. Similarly,
\begin{eqnarray}
&&\lim_{\lambda\to\infty}\lambda \int_0^t T_{A_{2_0}}(t-s)\int_0^a \pi(a, \tau)e^{-\lambda(a-\tau)}g_2(\tau, w(\cdot, \tau, s))d\tau ds \nonumber\\
&&=\lim_{\lambda\to\infty}\lambda \int_{\max\{t-a,0\}}^t \pi(a, a-t+s) \int_0^{a-t+s} \pi(a-t+s, \tau)e^{-\lambda(a-t+s-\tau)}g_2(\tau, w(\cdot, \tau, s))d\tau ds\nonumber\\
&&=\int_{\max\{t-a,0\}}^t \pi(a, a-t+s) g_2(a-t+s, w(\cdot, a-t+s, s))ds\\
&&
 = \begin{cases}
     \int_{0}^t\pi(a,a-s)g_2(a-s,w(\cdot, a-s, t-s))ds \ &\text{if}\ a> t, \nonumber\\
    \int_{0}^a\pi(a,a-s)g_2(a-s,w(\cdot, a-s, t-s))ds \ &\text{if}\ a<t.
\end{cases}
\end{eqnarray}

Combining \eqref{we1}-\eqref{we3}, we have
\begin{equation}\label{w11}
    w(\cdot, a, t)=
      \begin{cases}
    \pi(a, a-t)w_0(\cdot, a-t)+ \int_{0}^t\pi(a,a-s)g_2(a-s,w(\cdot, a-s, t-s))ds, \ &\text{if}\ a> t, \\
\pi(a, 0)g_1(u(\cdot, t-a), P(t-a))+    \int_{0}^a\pi(a,a-s)g_2(a-s,w(\cdot, a-s, t-s))ds, \ &\text{if}\ a<t.
\end{cases}
\end{equation}
Along a characteristic line $a-t=c$ with fixed $c>0$, by the first case in \eqref{w11}, $w$ satisfies
\begin{equation}\label{w22}
w(\cdot, t+c, t)=\pi(t+c, c) w_0(\cdot, c)+\int_0^t \pi(t+c, t+c-s)g_2(t+c-s, w(\cdot, t+c-s, t-s))ds.
\end{equation}
Noticing $w\in C([0, t_{\max}), Z)$ and  by the continuity assumptions of $\mu$ in (A3), equation \eqref{w22} implies that $w(\cdot, t+c, t)$ belongs to $C^1([0, t_{\max}), Y)$. Moreover,  differentiating \eqref{w22} with respect to $t$ gives
\begin{eqnarray}
\frac{d}{dt} w(\cdot, t+c, t)=-\mu(\cdot, t+c, P(t)) w(\cdot, t+c, t), \ \ \forall t\ge 0.
\end{eqnarray}
This implies that
$$
w(x, t+c, t)= w(x, c, 0) e^{-\int_0^t \mu(x, t+c-s, P(t-s))ds}, \ \ \forall t\ge 0, x\in\bar\Omega.
$$
Plugging in $c=a-t>0$, we obtain
$$
w(x,a, t)= w_0(x, a-t) e^{-\int_0^t \mu(x, a-s, P(t-s))ds}, \quad\forall  a>t\ge 0, x\in\bar\Omega.
$$
Similarly, using the  second case in \eqref{w11}, we have
$$
w(x,a, t)=\chi(x, P(t-a))e(x) u(x, t-a) e^{-\int_0^a
\mu(x, a-s, P(t-s))ds}, \quad\forall  t>a\ge 0, x\in\bar\Omega.
$$
\end{proof}

\subsection{Positivity of solutions}
In this subsection, we show that the semiflow $\{U(t)\}_{t\geq 0}$ in Theorem \ref{solution semiflow} induced by the solutions of \eqref{model} is nonnegative. Define a bounded linear operator $Q: \mathbb{X}_{0}\rightarrow \mathbb{X}$ by $Q\bm x=\bm x$ for any $\bm x=(u, (0_Y, w)) \in\mathbb{X}_0$.
Denote $Z_{+}:=L^{1}((0,a_{\max}),Y_{+})$,
$X_{+}:=Y_{+}\times Z_{+}$, $\mathbb{X}_{+}:=Y_{+}\times X_{+}$, and $\mathbb{X}_{0+}:=\mathbb{X}_{+}\cap \mathbb{X}_{0}=Y_{+}\times \{0_Y\}\times Z_{+}$, which are the positive cones of $Z$, $X$, $\mathbb{X}$ and $\mathbb{X}_0$, respectively.

\begin{lemma}\label{positivity:1}
Let $\mu_0\in L^\infty(\bar\Omega\times (0, a_{\max}))$ and $A$ be defined by \eqref{A}.
For any $\lambda>\max\{-\gamma, -\gamma-\underline\mu_0\}$ with $\gamma>0$, $\lambda\in \rho(A-\gamma Q)$ $($the resolvent set of $A-\gamma Q$$)$, and the following explicit formula holds for the resolvent of $A-\gamma Q$: for any $(\varpi,(y,\vartheta))\in \mathbb{X}$ and $(u, (0_Y, w)) \in \mathbb{X}_0$,  the following equality
\begin{equation}\label{resolvent positive}
\begin{gathered}(\lambda I-(A-\gamma Q))^{-1}\begin{pmatrix}
\varpi\\ \begin{pmatrix}
y\\ \vartheta
\end{pmatrix}
\end{pmatrix}
=\begin{pmatrix}
u\\ \begin{pmatrix}
0_Y\\ w
\end{pmatrix}
\end{pmatrix}
\end{gathered}
\end{equation}
holds if and only if
$$
u=((\lambda+\gamma)I-d\Delta)^{-1}\varpi,
$$
and
$$
w(a)=e^{-\int_{0}^{a}(\mu_0(l)+\lambda+\gamma) dl}y +\int_{0}^{a}e^{-\int_{s}^{a}(\mu_0(l)+\lambda+\gamma) dl}\vartheta(s)ds.
$$
In particular,  the linear operator $A-\gamma Q$ is resolvent positive in the sense that if $\lambda>0$ is large then
\begin{equation}\label{positivity_condition_1}
\begin{gathered}
(\lambda I-(A-\gamma Q))^{-1}\mathbb{X}_{+}\subset \mathbb{X}_{+}.
\end{gathered}
\end{equation}
\end{lemma}
\begin{proof}
For any $\lambda\in\mathbb{R}$, $(\varpi, (y,\vartheta))\in \mathbb{X}$ and $(u, (0_Y, w)) \in \mathbb{X}_0$,
 we have
\begin{align*}
& \qquad ~ \quad(\lambda I-(A-\gamma Q))^{-1}\begin{pmatrix}
\varpi\\ \begin{pmatrix}
y\\ \vartheta
\end{pmatrix}
\end{pmatrix}
=\begin{pmatrix}
u\\ \begin{pmatrix}
0_Y\\ w
\end{pmatrix}
\end{pmatrix}
\\  & \Longleftrightarrow \quad
(\lambda I-(A-\gamma Q))\begin{pmatrix}
u\\ \begin{pmatrix}
0_Y\\ w
\end{pmatrix}
\end{pmatrix}
=\begin{pmatrix}
\varpi\\ \begin{pmatrix}
y\\ \vartheta
\end{pmatrix}
\end{pmatrix}
\\   &\Longleftrightarrow \quad
\lambda \begin{pmatrix}
u\\ \begin{pmatrix}
0_Y\\ w
\end{pmatrix}
\end{pmatrix}
-\begin{pmatrix}
d \Delta u\\ A_{2}\begin{pmatrix}
0_Y\\ w
\end{pmatrix}
\end{pmatrix}
+\gamma \begin{pmatrix}
u\\ \begin{pmatrix}
0_Y\\ w
\end{pmatrix}
\end{pmatrix}
=\begin{pmatrix}
\varpi\\ \begin{pmatrix}
y\\ \vartheta
\end{pmatrix}
\end{pmatrix}
  \\   &\Longleftrightarrow \quad
\left\{
   \begin{aligned}
&\ \lambda u-d\Delta u+\gamma u=\varpi &\\
&\ \lambda \begin{pmatrix}
0_Y\\ w
\end{pmatrix}
-A_{2} \begin{pmatrix}
0_Y\\ w
\end{pmatrix}
+\gamma \begin{pmatrix}
0_Y\\ w
\end{pmatrix}
=\begin{pmatrix}
y\\ \vartheta
\end{pmatrix} &\\
   \end{aligned}
  \right.
   \\   &\Longleftrightarrow \quad
   \left\{
   \begin{aligned}
   &\ \lambda u-d\Delta u+\gamma u=\varpi &\\
    &\   w(0)=y &\\
    &\ \lambda w+\partial_a w+\mu_0 w+\gamma w=\vartheta.
    \end{aligned}
   \right.
   \end{align*}
It follows that
\begin{equation}\label{ph1}
u=(\lambda+\gamma-d\Delta)^{-1}\varpi.
\end{equation}
Similar to the proof of Lemma \ref{H-Y},
\begin{equation}\label{var1}
w(a)=e^{-\int_{0}^{a}(\mu_0(l)+\lambda+\gamma) dl}y +\int_{0}^{a}e^{-\int_{s}^{a}(\mu_0(l)+\lambda+\gamma) dl}\vartheta(s)ds.
\end{equation}
The above computations hold if  $\lambda>\max\{-\gamma, -\gamma-\underline\mu_0\}$. Clearly, \eqref{positivity_condition_1} follows from \eqref{ph1} and \eqref{var1}.
\end{proof}

\begin{lemma}\label{positivity:2}
For each $\xi>0$, there exists  $\gamma=\gamma(\xi)>0$ such that
$$
F(\bm x)+\gamma Q\bm x\geq 0_{\mathbb{X}}
$$
for any $\bm x\in \mathbb{X}_{0+}$ satisfying  $\|\bm x\|_{\mathbb{X}}\leq \xi.$
\end{lemma}
\begin{proof}
Let $\xi>0$ and $\bm x=(u, (0_Y, w))\in\mathbb{X}_{0+}$ satisfying $\|\bm x\|_{\mathbb{X}}\leq \xi$.
Choose $\gamma>0$ such that
$$
\gamma>\text{ess sup} \ \mu +\text{ess sup} \ \mu_0 +\|m\|_\infty+\|e\|_\infty+2\|c\|_\infty\xi.
$$
Then,
\begin{equation*}
   \begin{aligned}
 f(u,w)+\gamma u &=\int_{0} ^{a_{\max}}\beta(\cdot, a, \tilde w)w(a)da-(m+e)u-c u^{2}+\gamma u\geq 0_Y, \\
 g_1(u, \tilde w)&=\chi(\cdot, \tilde w)e u\ge 0_Y,\\
  g_2(\cdot,w)+\gamma w &=\mu_0 w-\mu(\cdot, \tilde w) w+\gamma w\geq 0_Z.
  \end{aligned}
  \end{equation*}
Thus, $F(\bm x)+\gamma Q\bm x\geq 0_{\mathbb{X}}$.  This completes the proof.
\end{proof}

Lemmas \ref{positivity:1} and \ref{positivity:2} show that positivity conditions in \cite[Assumption 4.1]{Magal2019} hold. Therefore, by \cite[Theorem 4.3]{Magal2019}, the  semiflow induced by the solutions of \eqref{model} is positive:

\begin{theorem}\label{Positivity}  Suppose that (A1)-(A3) hold. Then the semiflow $\{U(t)\}_{t\ge0}$ induced by  the solutions of \eqref{model} is nonnegative, that is, if $\bm x_0\in \mathbb{X}_{0+}$, then
$$
U(t)\bm x_0 \in \mathbb{X}_{0+},\quad \forall t\in (0,t_{\max}).
$$
\end{theorem}

\subsection{Monotonicity and comparison principle}
The following assumption will be imposed to obtain the monotonicity of the solutions of \eqref{model}.
\begin{itemize}
\item[(A4)] The following assumptions hold:
\begin{enumerate}
    \item[{\rm (i)}] The map   $w\to \int_0^{a_{\max}} \beta(\cdot, a, \tilde w) w(\cdot, a)da$ is an increasing function from $Z_+$ to $Y_+$;
    \item[{\rm (ii)}] For any $\xi>0$, there is $\gamma>0$ such that the map $w \to \gamma w-\mu(\cdot, \cdot, \tilde w) w$ is an increasing function from the set $\{ w\in Z_+: \ \|w\|_Z\le\xi\}$ to $Z_+$;
    \item[{\rm (iii)}] The map $P\to \chi(\cdot, P)$ is increasing from $[0, \infty)$ to $Y_+$.
\end{enumerate}
\end{itemize}
It is easy to check that the following lemma holds with assumption (A4).

\begin{lemma}\label{monotonicity}
Suppose that $(A1)$-$(A4)$ hold.
Then for any $\xi>0$, there is $\gamma>0$ such that
$$
0_{\mathbb{X}}\leq \bm x_{1}\leq \bm x_{2}\Rightarrow 0_{\mathbb{X}}\leq F(\bm x_{1})+\gamma Q \bm x_{1}\leq F(\bm x_{2})+\gamma Q \bm x_{2},
$$
whenever $\bm x_{1}, \bm x_{2}\in \mathbb{X}_{0}$ satisfies $\|\bm x_{1}\|_{\mathbb{X}}\leq \xi$ and $\|\bm x_{2}\|_{\mathbb{X}}\leq \xi$.
\end{lemma}

The following result is a direct consequence of \cite[Theorem 4.5]{Magal2019} and Lemma \ref{monotonicity}, which claims that the semiflow induced by the solutions of \eqref{model} is monotone under assumptions (A1)-(A4).

\begin{theorem}\label{Monotone Semiflow}
Suppose that $(A1)$-$(A4)$  hold. Then the semiflow $\{U(t)\}_{t\geq 0}$ induced by the solutions of \eqref{model} is nonnegative and monotone increasing. That is, for any $\bm x_{1}, \bm x_{2}\in \mathbb{X}_{0+}$,
$$0\leq \bm x_{1}\leq \bm x_{2}\Longrightarrow 0\leq U(t)\bm x_{1}\leq U(t)\bm x_{2},\quad \forall t\in (0,t_{\max}).$$
\end{theorem}

The next two results follow from \cite[Propositions 5.1-5.4]{Magal2019}, which give a method to construct upper and lower solutions of \eqref{model}.

\begin{proposition}\label{supper solution: integral form}\textbf{(Integral form)}
Suppose that $(A1)$-$(A4)$ hold. Let  $z\in C([0,t_{\max}),\mathbb{X}_{0+})$ and $\bm x_0\in\mathbb{X}_0$.
For any sufficiently large $\gamma>0$, if
$$
z(t)\geq (\le) T_{(A-\gamma Q)_{0}}(t)\bm x_0+(S_{(A-\gamma Q)}\diamond (F+\gamma Q)(z(\cdot)))(t),
$$
then
$$
z(t)\geq (\le) U(t)\bm x_0,\quad t\in [0,t_{\max}),
$$
where $\{U(t)\}_{t\ge0}$ is the  semiflow induced by the solutions of \eqref{model}.
\end{proposition}

\begin{proposition}\label{differential form upper solution}\textbf{(Differential form)}
Suppose that $(A1)$-$(A4)$ hold. Let $\bm x_0 \in \mathbb{X}_{0+}$ and   $\bm y\in C([0,t_{\max}),D(A))\cap C^{1}([0,t_{\max}),\mathbb{X}_{0+})$. If
\begin{equation*}
\left\{
  \begin{array}{ll}
      \frac{d\bm y(t)}{dt}\geq (\le)  A\bm y(t)+F(\bm y(t)),\quad \forall t\in [0,t_{\max}),   \\
     \bm y(0)=\bm x_0,   
  \end{array}
\right.
\end{equation*}
then we have
$$
\bm y(t)\geq (\le)  U(t)\bm x_0,\quad \forall t\in [0,t_{\max}),
$$
where $\{U(t)\}_{t\ge0}$ is the  semiflow induced by the solutions of \eqref{model}.
\end{proposition}


By \cite[Theorem 5.5]{Magal2019}, the following result holds.
\begin{theorem}\label{Increasing and Decreasing Solution}
Suppose that $(A1)$-$(A4)$ hold. Let $\{U(t)\}_{t\ge 0}$ be the semiflow induced by the solutions of \eqref{model} and  $\bm x_0\in D(A)\cap \mathbb{X}_{0+}$. Then the following statements hold:
\begin{enumerate}
    \item[{\rm (i)}]  If $A\bm x_0+F(\bm x_0)\leq 0_{\mathbb{X}}$, then the map $t\rightarrow U(t)\bm x_0$ is decreasing on $[0, t_{\max})$;
    \item[{\rm (ii)}] If $A\bm x_0+F(x_0)\geq 0_{\mathbb{X}}$, then the map $t\rightarrow U(t)\bm x_0$ is increasing on $[0, t_{\max})$.
\end{enumerate}
\end{theorem}

\begin{remark}
    Theorem \ref{Increasing and Decreasing Solution} gives conditions under which the solution of \eqref{model} is increasing or decreasing. Unfortunately, it does not tell us whether the solutions converge to an equilibrium as $t\to\infty$ if $t_{\max}=\infty$. We will show that if the $\omega$-limit set of $\bm x_0$ is compact, then $U(t)\bm x_0$ converges to an equilibrium as $t\to\infty$ in Section \ref{section_as}.
\end{remark}

\section{Boundedness and global existence}
We make the following assumption in order to establish the boundedness and global existence of the solutions of model \eqref{model}.

\begin{itemize}
\label{Hypothesis one}
\item[(A5)] $\bar{\chi}=\sup_{(x,P)\in {\Omega}\times \mathbb{R}_{+}}\chi(x,P)<\infty$ and
$\underline{\mu}:=\text{essinf} \ \mu>0$.
\end{itemize}
\begin{theorem}\label{global existence}
Suppose that $(A1)$-$(A3)$ and $(A5)$ hold.  Let $\{U(t)\}_{t\ge0}$ be the semiflow  induced by  the solutions of \eqref{model}. Then for any $\xi>0$ and $\bm x_0\in \mathbb{X}_{0+}$ with $\|\bm x_0\|_{\mathbb{X}}\le \xi$, there exists $M=M(\xi)>0$ such that
\begin{equation}\label{Mb}
\|U(t)\bm x_0\|_{\mathbb{X}}\le M, \ \ \forall t\ge 0.
\end{equation}
In particular, the nonnegative solution $U(t)\bm x_0$ of \eqref{model} exists for all $t\ge0$. Moreover, there is $N>0$ independent of $\bm x_0\in \mathbb{X}_0$ such that
\begin{equation}\label{Nb}
\limsup_{t\to\infty} \|U(t)\bm x_0\|_{\mathbb{X}}\le N.
\end{equation}
\end{theorem}
\begin{proof}
Let $\xi>0$ and $\bm x_0=(u_0, (0_Y, w_0)) \in \mathbb{X}_{0+}$ with $\|\bm x_0\|_{\mathbb{X}}\le \xi$.
By Theorems \ref{solution semiflow} and \ref{Positivity}, the solution $U(t)\bm x_0= (u(\cdot, t), (0_Y, w(\cdot, \cdot, t))) $ of \eqref{model} exists for $t\in [0, t_{\max})$ such that $U(t)\bm x_0\in \mathbb{X}_{0+}$ for all $t\ge 0$ and either $t_{\max}=\infty$ or $t_{\max}<\infty$ and $\lim_{t\to t_{\max}}\|U(t)\bm x_0\|_{\mathbb{X}}=\infty$.

We only consider the case $a_{\max}=\infty$, since  $a_{\max}<\infty$ can be proved similarly. By \eqref{wint}, we have
\begin{eqnarray}
     \mathbb{P}(x, t):=\int_0^{\infty} w(x, a, t)da&=&\int_0^t \chi(x, P(t-a))e(x) u(x, t-a) e^{-\int_0^a
\mu(x, a-s, P(t-s))ds} da \nonumber\\
&&+\int_t^\infty  w_0(x, a-t) e^{-\int_0^t \mu(x, a-s, P(t-s))ds} da  \label{Peq}
\end{eqnarray}
for all $t\in (0, t_{\max})$ and $x\in\bar\Omega$.
It follows that
\begin{eqnarray}
    \|\mathbb{P}(\cdot, t)\|_\infty &\le& \int_0^t \|\chi(\cdot, P(t-a))e(\cdot) u(\cdot, t-a) e^{-\int_0^a
\mu(\cdot, a-s, P(t-s))ds}\|_\infty da   \nonumber\\
    &&+\int_t^\infty  \|w_0(\cdot, a-t) e^{-\int_0^t \mu(\cdot, a-s, P(t-s))ds}\|_\infty da   \nonumber\\
    &\le& \bar\chi \bar e \sup_{t\in [0, t_{\max})} \|u(\cdot, t)\|_\infty\int_0^t  e^{-\underline\mu a} da +\int_t^\infty \|w_0(\cdot, a-t)\|_\infty e^{-\underline \mu t} da \nonumber\\
  &\le& \frac{\bar\chi\bar e}{\underline\mu}  \sup_{t\in [0, t_{\max})} \|u(\cdot, t)\|_\infty + \|w_0\|_Z \label{inf1}
\end{eqnarray}
for all $t\in [0, t_{\max})$. Note that
$$
\sup_{t\in [0, t_{\max})}\|B(\cdot, t)\|_\infty = \sup_{t\in [0, t_{\max})}\big\|\int_0^{a_{\max}}\beta(\cdot, a,P)w(\cdot,a, t)da\big\|_\infty \le \bar\beta \sup_{t\in [0, t_{\max})}\ \|\mathbb{P}(\cdot, t)\|_\infty,
$$
where  $\bar{\beta}:=\text{esssup} \ \beta$. 
By \eqref{uequ}, we have
$$
u_t-d\Delta u\le \bar\beta\sup_{t\in [0, t_{\max})}\ \|\mathbb{P}(\cdot, t)\|_\infty -\underline c u^2.
$$
By the comparison principle for parabolic equations, we have
\begin{equation}\label{inf2}
 \sup_{t\in [0, t_{\max})} \|u(\cdot, t)\|_\infty  \le \max\left\{\|u_0\|_\infty,  \ \sqrt{\frac{\bar\beta \sup_{t\in [0, t_{\max})}\ \|\mathbb{P}(\cdot, t)\|_\infty}{\underline c}}\right\}.
\end{equation}
Combining \eqref{inf1}-\eqref{inf2}, we have
$$
\sup_{t\in [0, t_{\max})}\ \|\mathbb{P}(\cdot, t)\|_\infty \le \max\left\{ \frac{\bar\chi\bar e}{\underline \mu} \|u_0\|_\infty+\|w_0\|_Z, \  \frac{1}{4}\left( \frac{\bar\chi\bar e}{\underline\mu} \sqrt{\frac{\bar\beta}{\underline c}} +  \sqrt{\frac{\bar\chi^2\bar e^2}{\underline\mu^2}\frac{\bar\beta}{\underline c}+4 \|w_0\|_Z }  \right)^2  \right\}=:M_1
$$
and
\begin{equation}\label{infuu}
\sup_{t\in [0, t_{\max})} \|u(\cdot, t)\|_\infty  \le \max\left\{\|u_0\|_\infty,  \ \sqrt{\frac{\bar\beta M_1}{\underline c}}\right\}=:M_2.
\end{equation}
By \eqref{wint} again, for all $t\in [0, t_{\max})$,
\begin{eqnarray}
\|w(\cdot, \cdot, t)\|_Z=\int_0^\infty \|w(\cdot, a ,t)\|_\infty da &\le& \int_0^t \|\chi(\cdot, P(t-a))e(x) u(\cdot, t-a) e^{-\int_0^a
\mu(\cdot, a-s, P(t-s))ds}\|_\infty da \nonumber\\
&&+\int_t^\infty \| w_0(a-t) e^{-\int_0^t \mu(x, a-s, P(t-s))ds} \|_\infty da \label{ww0}\\
&\le& \bar\chi \bar e M_2 \int_0^t e^{-\underline\mu a}da+\int_t^\infty e^{-\underline\mu t}\|w_0(\cdot, a-t )\|_\infty da \nonumber\\
&\le& \frac{\bar\chi \bar e M_2}{\underline\mu} + \|w_0\|_Z=:M_3.  \nonumber
\end{eqnarray}
 By \eqref{infuu} and \eqref{ww0}, we have $t_{\max}=\infty$, and the solution of \eqref{model} exists for all  $t\in [0, \infty)$ with
$$
\|U(t)\bm x\|_{\mathbb{X}}=\|u(\cdot, t)\|_\infty+\|w(\cdot, \cdot, t)\|_Z\le M:=M_2+M_3
$$
for all $t\in [0, \infty)$.

It remains to prove \eqref{Nb}. Changing integration variables in \eqref{Peq} and differentiating    with respect to $t$, we obtain
\begin{eqnarray}
     \partial_t \mathbb{P}(x, t) &=&  \partial_t\int_0^t \chi(x, P(t-a))e(x) u(x, t-a) e^{-\int_0^a
\mu(x, a-s, P(t-s))ds} da  \nonumber\\
&&+\partial_t\int_t^\infty  w_0(x, a-t) e^{-\int_0^t \mu(x, a-s, P(t-s))ds} da   \nonumber\\
   &=&  \partial_t\int_0^t \chi(x, P(a))e(x) u(x, a) e^{-\int_a^{t}
\mu(x, s-a, P( s))ds} da \nonumber\\
&&+\partial_t\int_0^\infty  w_0(x, a) e^{-\int_0^t \mu(x, a+s, P(s))ds} da \nonumber\\
&=& \chi(x, P(t))e(x) u(x, t) \nonumber\\
&&- \int_0^t \mu(x, t-a, P(t))\chi(x, P(a))e(x) u(x, a) e^{-\int_a^{t}
\mu(x, s-a, P(s))ds} da \nonumber\\
&&-\int_0^\infty \mu(x, a+t, P( t))  w_0(x, a) e^{-\int_0^t \mu(x, a+s, P(s))ds} da \nonumber\\
&=&  \chi(x, P(t))e(x) u(x, t)  \nonumber\\
&& - \int_0^t \mu(x, t, P(t))\chi(x, P(t-a))e(x) u(x, t-a) e^{-\int_0^{a}
\mu(x, a-s, P(t-s))ds} da \nonumber\\
&&-\int_t^\infty \mu(x, t, P(t))  w_0(x, a-t) e^{-\int_0^t \mu(x, a-s, P( t-s))ds} da \nonumber\\
&=& \chi(x, P(t))e(x) u(x, t) -\int_0^t \mu(x, a, P(t)) w(x,a, t)da-\int_t^\infty \mu(x, a, P(t)) w(x,a, t)da \nonumber.
\end{eqnarray}
Therefore,
\begin{equation}\label{Pd}
  \partial_t \mathbb{P}(x, t)= \chi(x, P(t))e(x) u(x, t) - \int_0^\infty \mu(x, a, P(t)) w(x,a, t)da.
\end{equation}
Note that \eqref{Pd} can be obtained formally from integrating the second equation of \eqref{model} with respect to variable $a$ in $(0, \infty)$. It follows that
$$
\partial_t \mathbb{P}(x, t)\le  \bar\chi \bar e u(x, t) - \underline\mu \mathbb{P}(x, t), \quad \forall x\in\bar\Omega, t\ge 0.
$$

Denote $u_\infty=\limsup_{t\to\infty} \|u(\cdot, t)\|_\infty<\infty$. Let $\epsilon>0$ be given. Then there exists $t_\epsilon>0$ such that $u(x, t)\le  u_\infty+\epsilon$ for all $x\in\bar\Omega$ and $t\ge t_\epsilon$. Hence,
$$
\partial_t \mathbb{P}(x, t)\le  \bar\chi \bar e (u_\infty+\epsilon) - \underline\mu \mathbb{P}(x, t), \quad \forall x\in\bar\Omega, t\ge t_\epsilon.
$$
Therefore, we have $\limsup_{t\to\infty} \mathbb{P}(x, t)\le \bar\chi\bar e (u_\infty+\epsilon)/\underline \mu$ uniformly for $x\in\bar\Omega$. Hence, there exists $t_\epsilon'>0$ such that $\mathbb{P}(x, t)\le \bar\chi\bar e (u_\infty+\epsilon)/\underline \mu+\epsilon$ for all $t\ge t_\epsilon'$ and $x\in\bar\Omega$. By \eqref{uequ}, we have
$$
\partial_t u-d\Delta u\le \bar\beta \Big[\bar\chi\bar e \frac{(u_\infty+\epsilon)}{\underline \mu}+\epsilon \Big] -\underline c u^2, \quad \forall x\in\bar\Omega, t\ge t_\epsilon'.
$$
This means
$$
u_\infty=\limsup_{t\to\infty} \|u(\cdot, t)\|_\infty \le \sqrt{\frac{\bar\beta}{\underline c} \Big[\bar\chi\bar e \frac{(u_\infty+\epsilon)}{\underline \mu}+\epsilon \Big]}.
$$
Taking $\epsilon\to 0$, we obtain
$$
u_\infty \le \frac{\bar\beta\bar\chi\bar e}{\underline c\underline\mu}=:N_1.
$$
By \eqref{ww0} and Fatou's lemma, we have
\begin{eqnarray*}
\limsup_{t\to\infty}\|w(\cdot, \cdot, t)\|_Z &\le& \limsup_{t\to\infty} \bar\chi\bar e\int_0^t \|u(\cdot, t-a)\|_\infty e^{-\underline\mu a} da +\limsup_{t\to\infty}\int_t^\infty \| w_0(\cdot, a-t) \|_\infty  e^{-\underline\mu t}da  \label{ww}\\
&\le&   \bar\chi\bar e\int_0^\infty \limsup_{t\to\infty}  \|u(\cdot, t-a)\|_\infty e^{-\underline\mu a} da + \limsup_{t\to\infty} e^{-\underline \mu t} \|w_0\|_Y\\
&\le&  \bar\chi\bar e\int_0^\infty \frac{\bar\beta\bar\chi\bar e}{\underline c\underline\mu} e^{-\underline\mu a} da + 0\\
&=& \frac{\bar\beta\bar\chi^2\bar e^2}{\underline c\underline\mu^2}=:N_2.
\end{eqnarray*}
Hence, \eqref{Nb} holds with $N:=\max\{N_1, N_2\}$.
\end{proof}

\section{Asymptotic smoothness}\label{section_as}

In this section, we show that the semiflow $\{U(t)\}_{t\geq 0}$ induced by the solutions of  system \eqref{model} is asymptotically smooth (see \cite{Hale1988} for the definition of asymptotic smoothness of a semiflow). 


We will need the following sufficient condition for asymptotic smoothness.

\begin{lemma}\label{asymptotically smooth theory}\text{(\cite[Lemma 3.2.3]{Hale1988})}
Let $X$ be a Banach space and $T(t):X\to X$, $t\ge0$, be a strongly continuous semigroup.  The semiflow $\{T(t)\}_{t\ge0}$ is asymptotically smooth if there are two maps $T_{1}, T_{2}: \mathbb{R}_{+}\times X\rightarrow X$ such that $T=T_1+T_2$ and  the following conditions hold:

\begin{enumerate}
    \item[{\rm (i)}] There is a continuous function $k:\mathbb{R}_+\times \mathbb{R}_+\to \mathbb{R}_+$ such that $k(t, \xi)\to 0$ as $t\to\infty$ and $\|T_1(t)\bm\psi\|_X\le k(t, \xi)$ for any $\xi>0$ and $\bm\psi\in X$ with $\|\bm\psi\|\le \xi$;
    \item[{\rm (ii)}] $T_{2}$ is  completely continuous, i.e.,  $T(t)B$ has compact closure in $X$ and $\{T(s) B: \ 0\le s\le t\}$ is bounded for any bounded set $B\subset X$ and $t>0$.
\end{enumerate}
\end{lemma}

To verify condition (ii) in Lemma \ref{asymptotically smooth theory}, we use the following Kolmogorov$'$s compactness criterion in $Z=L^{1}((0,a_{\max}),Y)$.

\begin{lemma}\label{compactness criterion}\text{(\cite[Theorem A.1]{gutman1982compact})}
Let $\mathbb{K}\subset Z:= L^{1}((0,a_{\max}),Y)$. The set $\mathbb{K}$ is relatively compact in $Z$ if and only if:
\begin{enumerate}
\item[{\rm (i)}] $\sup_{g\in \mathbb{K}}\|g\|_{Z}<\infty$;

\item[{\rm (ii)}] $\lim_{N\rightarrow a_{\max}}\int_{N}^{a_{\max}}\|g(a)\|_{Y}da=0$ uniformly for $g\in \mathbb{K}$;

\item[{\rm (iii)}] $\lim_{h\rightarrow 0}\int_{0}^{a_{\max}}\|g(a+h)-g(a)\|_{Y}da=0$ uniformly for $g\in \mathbb{K}$;

\item[{\rm (iv)}] For every $\epsilon>0$, there exists a compact set $R_{\epsilon}\subset Y$ such that for every $g\in \mathbb{K}$ there exists a set $Q_{g,\epsilon}$ with Lebesgue measure $\mu(Q_{g,\epsilon})<\epsilon$ and
$g(a)\in R_{\epsilon}$ for any $a\in [0,a_{\max}) \backslash Q_{g,\epsilon}$.
\end{enumerate}
 \end{lemma}

We will also need the following result, which follows from \cite[Lemma I.5.2]{Engel2000} and \cite[Proposition I.5.3]{Engel2000}.
\begin{lemma}\label{lemma_uc}
    Let $\{T(t)\}_{t\ge 0}$ be a strongly continuous semigroup on Banach space $B$. Then the map $L\times B_1 \ni (t, x)\to T(t)x$ is uniformly continuous for any compact sets $L\subseteq [0, \infty)$ and $ B_1\subseteq B$.
\end{lemma}

The main result in this section is the following theorem about the asymptotical smoothness of the semiflow.
\begin{theorem}\label{asymptotically smooth}
Suppose $(A1)$-$(A3)$ and $(A5)$ hold. Let $\{U(t)\}_{t\ge 0}$ be the semiflow induced by the solutions of \eqref{model}. Then $\{U(t)\}_{t\ge0}$ is asymptotically smooth.
\end{theorem}
\begin{proof}
In order to apply Lemma \ref{asymptotically smooth theory}, we decompose the  semiflow $\{U(t)\}_{t\ge0}$ as $U(t)=U_{1}(t)+U_{2}(t)$ with
$$
U_{1}(t)\bm x_0=(0_Y, (0_Y, w_{1}(\cdot,\cdot,t)))  \quad \text{and} \quad U_{2}(t)\bm x_{0}=(u(\cdot,t), (0_Y, w_{2}(\cdot,\cdot,t))) ,
$$
where $\bm x_0=(u_{0}, (0_Y, w_{0})) \in \mathbb{X}_{0+}$, $U(t)\bm x_0=(u(\cdot, t), (0_Y, w(\cdot, \cdot, t))) $, and $w=w_1+w_2$. Here, $w_1$ and $w_2$ are given by
\begin{equation*}
w_{1}(x,a,t)=
  \left\{
   \begin{array}{lll}
    w_0(x, a-t) e^{-\int_0^t \mu(x, a-s, P(t-s))ds}, \quad & \text{if}\ a>t, \\
0_Y, \quad ~ &\text{if}\ a<t,
   \end{array}
  \right.
  \end{equation*}
  and
  \begin{equation*}
w_{2}(x,a,t)=
  \left\{
   \begin{array}{lll}
0_Y, \quad & \text{if}\ a>t, \\
\chi(x, P(t-a))e(x) u(x, t-a) e^{-\int_0^a
\mu(x, a-s, P(t-s))ds}, \quad ~ &\text{if}\ a<t.
   \end{array}
  \right.
  \end{equation*}

If $a_{\max}=\infty$, we have that
\begin{eqnarray*}
 \|U_1(t)\bm x_0\|_{\mathbb{X}}&=&\|w_1(\cdot, \cdot, t)\|_Z\\
 &=&\int_t^\infty \|w_0(\cdot, a-t) e^{-\int_0^t \mu(\cdot, a-s, P(t-s))ds}\|_\infty da\\
 &\le&\int_t^\infty \|w_0(\cdot, a-t)\|_\infty e^{-\underline\mu t}da=e^{-\underline\mu t}\|w_0\|_Z.
\end{eqnarray*}
So if $\|\bm x_0\|_{\mathbb{X}}\le \xi$ for some $\xi>0$, then $ \|U_1(t)\bm x_0\|_{\mathbb{X}}\le e^{-\underline\mu t}\xi=: k(t, \xi)$. Since $k(t, \xi)\to 0$ as $t\to\infty$, $U_1(t)$ satisfies condition (i)  in Lemma \ref{asymptotically smooth theory}. If $a_{\max}<\infty$, then $w_1(\cdot, \cdot, t)=0_Z$ for all $t> a_{\max}$ and condition (i)  in Lemma \ref{asymptotically smooth theory} holds trivially.

It remains to prove that $U_2(t)$ satisfies condition (ii)  in Lemma \ref{asymptotically smooth theory}. Let $B\subset \mathbb{X}_0$ be a bounded set and $t>0$ be fixed. By Theorem \ref{global existence}, $\{U(s)B: \ 0\le s\le t\}$ is bounded in $\mathbb{X}$. By the well-known parabolic estimates, embedding theorems and Theorem \ref{global existence},  $\{u(\cdot, t):\ \bm x_0\in B\}$ is precompact in $Y$. Therefore, it suffices to prove that $\{w_2(\cdot, \cdot, t):\ \bm x_0\in B\}$ is precompact in $Z$.

We will apply Lemma \ref{compactness criterion} to prove that $\{w_2(\cdot, \cdot, t):\ \bm x_0\in B\}$ is precompact in $Z$. We will only consider the case $a_{\max}=\infty$ as the case  $a_{\max}<\infty$ is similar and simpler. By Theorem  \ref{global existence}, there is $M>0$ such that
\begin{equation}\label{pB}
    \|u(\cdot, t)\|_\infty\le M \quad \text{and} \quad \|w(\cdot,\cdot, t)\|_Z\le M, \quad \forall \bm x_0\in B, t\ge 0.
\end{equation}
Condition (i) of Lemma \ref{compactness criterion} immediately follows  from \eqref{pB}.  If $N\ge t$, then
$$
\int_N^\infty \|w_2(\cdot, a, t)\|_\infty da=\int_N^\infty 0\, da =0.
$$
Therefore, condition (ii) of Lemma \ref{compactness criterion} holds.

Next, we verify condition (iii) in Lemma \ref{compactness criterion}, i.e., for any $t>0$,
\begin{equation}\label{c3}
    \lim_{h\to 0} \int_0^\infty \|w_2(\cdot, a+h, t)-w_2(\cdot, a, t)\|_\infty da=0\quad \text{uniformly for} \ \bm x_0\in B.
\end{equation}
We assume $h>0$, and  the case $h< 0$ can be considered similarly. Fix $t>0$. We note that
\begin{eqnarray*}
 \int_0^\infty \|w_2(\cdot, a+h, t)-w_2(\cdot, a, t)\|_\infty da  &=& \int_0^{t-h}  \|w_2(\cdot, a+h, t)-w_2(\cdot, a, t)\|_\infty da\\
 &&+\int_{t-h}^t \|w_2(\cdot, a, t)\|_\infty da=: I_1+I_2.
\end{eqnarray*}
To estimate $I_1$, by \eqref{pB}, we have
\begin{eqnarray*}
    I_1&=&\int_0^{t-h}\Big\|\chi(\cdot, P(t-a-h))e u(\cdot, t-a-h) e^{-\int_0^{a+h}
\mu(\cdot, a+h-s, P(t-s))ds}\\
&&-\chi(\cdot, P(t-a))e u(\cdot, t-a) e^{-\int_0^a
\mu(\cdot, a-s, P(t-s))ds}\Big\|_\infty da\\
&\le&C\int_0^{t-h}\Big\|\chi(\cdot, P(t-a-h))- \chi(\cdot, P(t-a))\Big\|_\infty da + C\int_0^{t-h}\|u(\cdot, t-a-h)-u(\cdot, t-a)\|_\infty da \\
&&+C\int_0^{t-h}\Big\|e^{-\int_0^{a+h}
\mu(\cdot, a+h-s, P(t-s))ds}- e^{-\int_0^a
\mu(\cdot, a-s, P(t-s))ds}\Big\|_\infty da\\
&=&C(J_1+J_2+J_3).
\end{eqnarray*}
Here and after, $C$ is a constant depending on $M$.

To estimate $J_1$, integrating \eqref{Pd} over $\Omega$, we obtain that
$$
 | \partial_t P(t)|\le (\bar\chi \bar e   + \bar\mu)|\Omega| M,\quad\forall t\ge 0\  \text{and} \ \bm x_0\in B.
$$
Hence, we have
\begin{equation}
    | P(t+h)-P(t) |\le h (\bar\chi \bar e   + \bar\mu)|\Omega| M,\quad\forall t\ge 0\  \text{and} \ \bm x_0\in B.
\end{equation}
It follows that
\begin{equation}\label{J1}
    J_1\le C \int_0^{t-h}|P(t-a-h)- P(t-a)| da\le Cth, \quad \forall \bm x_0\in B.
\end{equation}

To estimate $J_2$, we note that, for any $\epsilon>0$, by \eqref{pB} and the well-known parabolic estimates and embedding theorems, the following set is precompact in $Y$:
$$
O_{\epsilon, B}:=\{u(\cdot, s):  \ \epsilon\le s\le t \ \text{and}\ \bm x_0\in B\}.
$$
Applying variation of constant formula to the first equation of \eqref{model}, we obtain
$$
u(\cdot, t-a)=  T_{A_1}(h) u(\cdot, t-a-h)+\int_{t-a-h}^{t-a} T_{A_1}(t-a-s) f(s) ds, \quad \forall a\in [0, t-h],
$$
where $f(s)=B(s)-(m+e)u-cu^2$. Since $\{T_{A_1}(t)\}_{t\ge0}$ is strongly continuous, by Lemma \ref{lemma_uc}, $\lim_{h\to 0} T_{A_1}(h)v=v$ uniformly for $v$ in compact subsets of $Y$.
Hence, by the compactness of $O_{\epsilon, B}$ and \eqref{pB}, we have
\begin{eqnarray*}
    \|u(\cdot, t-a)-u(\cdot, t-a-h)\|_\infty &\le& \|T_{A_1}(h) u(\cdot, t-a-h)-u(\cdot, t-a-h)\|_\infty \\
    &&+ \Big\|\int_{t-a-h}^{t-a} T_{A_1}(t-a-s) f(s) ds\Big\|_\infty\\
    &\le& C_\epsilon h + Ch
\end{eqnarray*}
for all $0\le a\le t-h-\epsilon$ and $\bm x_0\in B$, where $C_\epsilon$ is depending on $\epsilon$. It follows that
\begin{eqnarray}
    J_2&=& \int_0^{t-h-\epsilon}\|u(\cdot, t-a-h)-u(\cdot, t-a)\|_\infty da+\int_{t-h-\epsilon}^{t-h}\|u(\cdot, t-a-h)-u(\cdot, t-a)\|_\infty da \nonumber\\
    &\le&t(C_\epsilon +C)h+2M\epsilon  \label{J2}
\end{eqnarray}
 for all $\bm x_0\in B$.

 To estimate $J_3$, by the assumptions on $\mu$,
 \begin{eqnarray}
   J_3&\le &  \int_0^{t-h}\Big\|{\int_0^{a+h}
\mu(\cdot, a+h-s, P(t-s))ds}- {\int_0^a
\mu(\cdot, a-s, P(t-s))ds}\Big\|_\infty da \nonumber\\
&\le&\int_0^{t-h} \Big\| \int_a^{a+h} \mu(\cdot, a+h-s, P(t-s)) ds\Big\|_\infty da\nonumber\\
&&+\int_0^{t-h}\int_0^a \|\mu(\cdot, a+h-s, P(t-s))- \mu(\cdot, a-s, P(t-s))\|_\infty ds da\nonumber \\
&\le&\bar\mu t h + \int_0^{t-h}\int_0^a Ch ds da\le C(t+t^2)h \label{J3}
 \end{eqnarray}
 for all $\bm x_0\in B$.

 Combining \eqref{J1}, \eqref{J2} and \eqref{J3}, we get
 \begin{equation}\label{I1}
     I_1\le C(t+t^2)h+C_\epsilon t h+C\epsilon, \quad \forall \bm x_0\in B.
 \end{equation}
Now, we estimate $I_2$.  By \eqref{pB},
$$
\|w_2(\cdot, a, t)\|_\infty \le \bar\chi \bar e M.
$$
Hence,
\begin{equation}\label{I2}
    I_2=\int_{t-h}^t \|w_2(\cdot, a, t)\|_\infty da\le \bar\chi \bar e M h, \quad\forall \bm x_0\in B.
\end{equation}
Combining \eqref{I1} and \eqref{I2}, since $\epsilon>0$ was arbitrary, \eqref{c3} holds. This verifies condition (iii) in Lemma \ref{compactness criterion}.

Finally, we verify condition (iv) in Lemma \ref{compactness criterion}. By the definition of $w_2$, it suffices to show that, for fixed  $t>0$ and any $\epsilon>0$, the following set  is precompact in $Y$:
$$
W_{\epsilon, B}:=\{w_2(\cdot, a, t): \ 0<a<t-\epsilon \ \text{and} \ \bm x_0\in B\}.
$$
By \eqref{pB}, there exists $P^*>0$ such that $0\le P(t)\le P^*$ for all $t\ge 0$ and $\bm x_0\in B$. Since the continuous image of a compact set is compact, the following set
$$
\{e \chi(\cdot, P) e^{-\int_0^a
\mu(\cdot, a-s, P)ds}: \ 0\le P\le P^*\}
$$
is compact in $Y$. Since $O_{\epsilon, B}$ is precompact, $W_{\epsilon, B}$ is precompact in $Y$. This completes the proof.
\end{proof}

By the well-known dynamical system theory \cite{Hale1988, zhao2017dynamical}, combining Theorems \ref{global existence} and \ref{asymptotically smooth}, we have the following result:

\begin{corollary}\label{corollary_attractor}
    Suppose $(A1)$-$(A3)$ and $(A5)$ hold. Let $\{U(t)\}_{t\ge 0}$ be the semiflow induced by the solutions of \eqref{model}. Then $\{U(t)\}_{t\ge 0}$ has a global attractor in $\mathbb{X}$.
\end{corollary}

Now, we know that the $\omega$-limit set of $\bm x_0\in\mathbb{X}_0$ is compact. So we can improve Theorem \ref{Increasing and Decreasing Solution} to obtain  the convergence of  $U(t)\bm x_0$ to an equilibrium as $t\to\infty$.
\begin{theorem}\label{Theorem_c}
Suppose that $(A1)$-$(A5)$ hold. Let $\{U(t)\}_{t\ge 0}$ be the semiflow induced by the solutions of \eqref{model} and  $\bm x_0\in D(A)\cap \mathbb{X}_{0+}$. Then the following statements hold:
\begin{enumerate}
    \item[{\rm (i)}]  If $A\bm x_0+F(\bm x_0)\leq 0_{\mathbb{X}}$, then the map $t\rightarrow U(t)\bm x_0$ is decreasing on $[0, \infty)$, and  $U(t)\bm x_0\to \bm x^*$ in $\mathbb{X}$ as $t\to\infty$, where $\bm x^*\in  D(A)\cap \mathbb{X}_{0+}$ satisfies $A\bm x^*+F(\bm x^*)= 0_{\mathbb{X}}$;

    \item[{\rm (ii)}]  If $A\bm x_0+F(\bm x_0)\ge 0_{\mathbb{X}}$, then the map $t\rightarrow U(t)\bm x_0$ is increasing on $[0, \infty)$, and  $U(t)\bm x_0\to \bm x^*$ in $\mathbb{X}$ as $t\to\infty$, where $\bm x^*\in  D(A)\cap \mathbb{X}_{0+}$ satisfies $A\bm x^*+F(\bm x^*)= 0_{\mathbb{X}}$.
\end{enumerate}
\end{theorem}
\begin{proof}
 (i) Suppose that $\bm x_0\in D(A)\cap \mathbb{X}_{0+}$ satisfies $A\bm x_0+F(\bm x_0)\leq 0_{\mathbb{X}}$. By Theorem \ref{global existence}, the solution $U(t)\bm x_0\in \mathbb{X}_{0_+}$ is bounded and exists globally. By Corollary \ref{corollary_attractor}, the $\omega-$limit set of $U(t)\bm x_0$ is compact. By Theorem \ref{Increasing and Decreasing Solution}, $U(t)\bm x_0$ is decreasing in $[0, \infty)$. Hence, $U(t)\bm x_0\to \bm x^*$ in $\mathbb{X}$ for some $\bm x^*\in \mathbb{X}_{0+}$ as $t\to\infty$.

 By Theorem \ref{solution semiflow}, $U(t)\bm x_0$ satisfies $\int_{0}^{t}U(s)\bm x_0 ds\in D(A)$  and
\begin{equation}\label{integratedP}
 \begin{aligned}
U(t)\bm x_0=\bm x_0+A\int_{0}^{t}U(s)\bm x_0 ds+\int_{0}^{t}F(U(s)\bm x_0)ds, \quad \forall t>0.
   \end{aligned}
  \end{equation}
Let $\bm x_n=\int_{n-1}^n U(s)\bm x_0 ds$ for any $n\ge 1$. Then,
$$
\bm x_n=\int_0^n U(s) \bm x_0 ds-\int_0^{n-1} U(s) \bm x_0 ds \in D(A), \quad\forall n\ge 1.
$$
Moreover, by \eqref{integratedP},
$$
U(n)\bm x_0-U(n-1)\bm x_0=A\bm x_n+\int_{n-1}^n F(U(s)\bm x_0)ds, \quad\forall n\ge1.
$$
Since $U(n)\bm x_0\to \bm x^*$ in $\mathbb{X}$ as $n\to\infty$, we have
$$
\lim_{n\to\infty}A\bm x_n=-\lim_{n\to\infty} \int_{n-1}^n F(U(s)\bm x_0)ds=-F(\bm x^*).
$$
Since the operator $A: D(A)\subset\mathbb{X}\to \mathbb{X}$ is closed, we must have $\bm x^*\in D(A)$ and $A\bm x^*=-F(\bm x^*)$.

The proof of (ii) is similar to (i), so it is omitted.
\end{proof}

\section{The net reproductive rate}\label{threshold value}

In this section, we define the net reproductive rate $\mathcal{R}_0$ for the system \eqref{model}  and study the properties of  $\mathcal{R}_0$.  Biologically, $\mathcal{R}_0$ can be thought as the average number of offspring produced by a single individual. It is obvious that system \eqref{model} always has trivial steady state $E_{0}=(0_Y,(0_Y, 0_Z)) $. To define $\mathcal{R}_0$, we consider the linearized system of \eqref{model} at $E_{0}$:
  \begin{equation}\label{linearized model}
  \left\{
   \begin{array}{lll}
\partial_{t} u=d\Delta u +\displaystyle\int_{0} ^{a_{\max}}\beta(x,a,0)w(x,a,t)da-(m(x)+e(x))u, \quad &x\in\Omega, t>0, \\
\partial_{t}w+\partial_{a}w=-\mu(x,a, 0)w ,\quad &x\in\Omega, t>0,a\in (0,a_{\max}), \\
\partial_\nu u=0,~\quad &x\in \partial \Omega, t>0, \\
 w(x,0,t)=\chi(x,0)e(x)u(x,t),  \quad &x\in\Omega, t>0, \\
 u(x,0)=u_{0}(x), \quad &x\in\Omega, \\
 w(x,a,0)=w_{0}(x,a),\quad  &x\in\Omega, a\in (0, a_{\max}).
   \end{array}
  \right.
  \end{equation}

Throughout this section, we suppose $m, e, \chi(\cdot, 0) \in Y$ with $m(x), e(x)>0$ for all $x\in\bar\Omega$, and $\mu(\cdot, \cdot, 0), \beta(\cdot, \cdot, 0)\in L^\infty((0, a_{\max}), Y_+)\cap C([0, a_{\max}), Y_+)$ with $\underline \mu:=\text{essinf} \ \mu(\cdot, \cdot, 0)>0$.

\subsection{Definition of $\mathcal{R}_0$}
For system \eqref{linearized model}, we define linear operators $\mathcal{L}$, $\mathcal{B}$ and $\mathcal{F}$ on $\mathbb{X}$ as follows:
\begin{equation}\label{linear operator}
   \begin{aligned}
\mathcal{L}\bm x:=\mathcal{B}\bm x+\mathcal{F}\bm x, \quad\forall \bm x=(u, (0_Y, w)) \in D(\mathcal{L}),
   \end{aligned}
  \end{equation}
where
\begin{equation*}
\begin{gathered}\mathcal{B}\bm x
=\begin{pmatrix}
d\Delta u-(e+m)u \\
-w(\cdot,0)+\chi(\cdot,0)eu\\
-w_{a}-\mu(\cdot,\cdot,0)w
\end{pmatrix}, \quad\forall \bm x=(u, (0_Y, w)) \in D(\mathcal{B}),
\end{gathered}
\end{equation*}
and
\begin{equation*}
\begin{gathered}\mathcal{F}\bm x
=\begin{pmatrix}
\int_{0} ^{a_{\max}}\beta(\cdot,a,0) w(\cdot, a) da \\
0_Y\\
0_Z
\end{pmatrix}, \quad\forall \bm x=(u, (0_Y, w)) \in\mathbb{X}.
\end{gathered}
\end{equation*}
Here, $D(\mathcal{L})=D(\mathcal{B})=Y\times \{0_{Y}\}\times W^{1,1}((0,a_{\max}),Y)$.
Clearly, $\mathcal{L}$ is a positive perturbation of $\mathcal{B}$.

The following result says that  $\mathcal{B}$ is resolvent-positive (see \cite{Thieme2009} for the definition of resolvent-positive operators) with a negative spectral bound.
\begin{lemma}\label{lemmaB}
    The operator $\mathcal{B}: D(\mathcal{B})\subset\mathbb{X}\to\mathbb{X}$ is resolvent-positive with spectral bound $s(\mathcal{B})=\sup\{Re(\lambda):\lambda\in \sigma(\mathcal{B})\}<0$, where $\sigma(\mathcal{B})$ denotes the spectrum of $\mathcal{B}$.
\end{lemma}
\begin{proof}
Let $(\varpi, (y,\vartheta))\in \mathbb{X}$ and $(u,(0_Y,w))\in D(\mathcal{B})$. Then, for any  $\lambda\in\mathbb{C}$, one can check that
\begin{equation*}
(\lambda I-\mathcal{B})^{-1}\begin{gathered}\begin{pmatrix}
\varpi\\ \begin{pmatrix}
y\\ \vartheta
\end{pmatrix}
\end{pmatrix}
=\begin{pmatrix}
u \\
\begin{pmatrix}
0_Y\\ w
\end{pmatrix}
\end{pmatrix}
\end{gathered}
\end{equation*}
holds if and only if
\begin{equation}\label{exact expression}
  \left\{
   \begin{aligned}
&\ w(x,a)=e^{-\int_{0}^{a}(\mu(x,l,0)+\lambda) dl}(y+\chi(x,0)e(x)u(x)) +\int_{0}^{a}e^{-\int_{s}^{a}(\mu(x,l,0)+\lambda) dl}\vartheta(x,s)ds,&\\
&\ u=(\lambda +m+e-d\Delta)^{-1}\varpi.
   \end{aligned}
  \right.
  \end{equation}
Since $(\lambda +m+e-d\Delta)^{-1}$ is a positive operator, by \eqref{exact expression}, $(\lambda I-\mathcal{B})$ is invertible and  $(\lambda I-\mathcal{B})^{-1}$ is positive for any $\lambda>0$. Hence, $\mathcal{B}$ is resolvent-positive.

By our assumption,  $\mu(x, a, 0)\ge \underline \mu>0$ for all $x\in\bar\Omega$ and $a\in (0, a_{\max})$. Therefore, from \eqref{exact expression}, if $Re(\lambda)>\max\{-\underline\mu/2, -\min_{x\in\bar\Omega}(m+e)\}$, $(\lambda I-\mathcal{B})^{-1}$ is a bounded linear operator. Hence,  $s(\mathcal{B})<0$.
\end{proof}

With Lemma \ref{lemmaB}, we are able to show that $\mathcal{L}$ is also resolvent-positive.
\begin{lemma}\label{lemmaL}
The operator $\mathcal{L}: D(\mathcal{L})\subseteq\mathbb{X}\to\mathbb{X}$ is resolvent-positive.
\end{lemma}
\begin{proof}
Note that $\mathcal{L}=\mathcal{B}+\mathcal{F}$ with $\mathcal{B}$ being resolvent-positive, $s(\mathcal{B})<0$, and $\mathcal{F}$ being positive.  By \cite[Theorem 3.4]{Thieme2009}, $\mathcal{L}$ is resolvent-positive if and only if the spectral radius, $r(\mathcal{F}(\lambda I-\mathcal{B})^{-1})$, of the operator $\mathcal{F}(\lambda I-\mathcal{B})^{-1}$ satisfies $r(\mathcal{F}(\lambda I-\mathcal{B})^{-1})<1$
for some $\lambda>s(\mathcal{B})$.

By \eqref{exact expression}, one can see that
\begin{equation*}
\begin{aligned}
& \left\|(\lambda I-\mathcal{B})^{-1}\begin{pmatrix}
\varpi\\ \begin{pmatrix}
y\\ \vartheta
\end{pmatrix}
\end{pmatrix}\right\|_{\mathbb{X}}\leq\|u\|_{\infty}+\int_{0}^{a_{\max}}\|w(\cdot,a)\|_{\infty}da\rightarrow 0, ~\text{as}~\lambda\rightarrow \infty.
\end{aligned}
\end{equation*}
Since $\mathcal{F}$ is a bounded linear operator, we have
$$r(\mathcal{F}(\lambda I-\mathcal{B})^{-1})\leq \|\mathcal{F}\|_{\mathfrak{L}(\mathbb{X})}\|(\lambda I-\mathcal{B})^{-1}\|_{\mathfrak{L}(\mathbb{X})}<1,$$
for   sufficiently large $\lambda$. Therefore,  operator $\mathcal{L}$ is resolvent-positive.
\end{proof}

Following \cite{Thieme2009}, the  \emph{net reproductive rate} $\mathcal{R}_{0}$ of system \eqref{model} is defined as the spectral radius of $-\mathcal{F}\mathcal{B}^{-1}$, i.e.,
\begin{equation}\label{basic reproduction number}
\begin{aligned}
& \mathcal{R}_{0}=r(-\mathcal{F}\mathcal{B}^{-1}).
\end{aligned}
\end{equation}

By \cite[Theorem 3.5]{Thieme2009}, using Lemmas \ref{lemmaB} and \ref{lemmaL}, we have the following conclusion.
\begin{proposition}\label{prop1}
 Let $\mathcal{R}_{0}$ be defined by \eqref{basic reproduction number}. Then $s(\mathcal{L})$ has the same sign as $\mathcal{R}_{0}-1.$
\end{proposition}

We now compute $\mathcal{R}_0$ and show that it is a principal eigenvalue of some elliptic problem.
\begin{proposition}
  Let $\mathcal{R}_{0}$ be defined by \eqref{basic reproduction number}.  If  $\mathcal{R}_0>0$,  $\lambda=1/\mathcal{R}_0$ is an eigenvalue of the following eigenvalue problem that corresponds with a positive eigenfunction:
 \begin{equation}\label{Reig}
  \left\{
   \begin{array}{lll}
\lambda \phi \chi(x,0)e(x) \int_{0} ^{a_{\max}}\beta(x,a,0)e^{-\int_{0}^{a}\mu(x,l,0)dl}da= (m+s-d\Delta)\phi,&x\in\Omega,\\
\partial_\nu\phi=0, &x\in\partial\Omega.
   \end{array}
  \right.
  \end{equation}
\end{proposition}
\begin{proof}
Let $(\varpi, (y,\vartheta)) \in \mathbb{X}$. Then we have $\mathcal{B}^{-1}(\varpi, (y,\vartheta)) =(u, (0_Y, w)) \in \mathbb{X}_0$, where $u$ and $w$ are given by \eqref{exact expression} with $\lambda=0$. It follows that
\begin{eqnarray}\label{explicit formula of BRN} \nonumber
-\mathcal{F}\mathcal{B}^{-1}
\begin{pmatrix}
\varpi\\
y\\
\vartheta
\end{pmatrix}
&=&\mathcal{F}\begin{pmatrix}
u\\
0_Y\\
w
\end{pmatrix}
=\begin{pmatrix}
\int_{0} ^{a_{\max}}\beta(\cdot,a,0) w(\cdot,a)da \\
0_Y\\
0_Z
\end{pmatrix}\\
&=&\begin{pmatrix}
R_{1}\varpi+R_{2}y+R_{3}\vartheta \\
0\\0
\end{pmatrix}
=\begin{pmatrix}
R_{1} & R_{2} & R_{3}\\ 0&0&0 \\ 0&0&0
\end{pmatrix}
\begin{pmatrix}
\varpi \\ y\\ \vartheta
\end{pmatrix}, \nonumber
\end{eqnarray}
 where
 \begin{equation}\label{R1}
\begin{aligned}
 R_{1}\varpi&=\chi(\cdot,0)e\int_{0} ^{a_{\max}}\beta(\cdot,a,0)e^{-\int_{0}^{a}\mu(\cdot,l,0)dl}da(m+s-d\Delta)^{-1}\varpi, \\
R_{2}y&=\int_{0} ^{a_{\max}}\beta(\cdot,a,0)e^{-\int_{0}^{a}\mu(\cdot,l,0)dl}y da, \\
R_{3}\vartheta &=\int_{0} ^{a_{\max}}\beta(\cdot,a,0)\int_{0}^{a}e^{-\int_{s}^{a}\mu(\cdot,l,0)dl}\vartheta(\cdot,s)dsda.
\end{aligned}
\end{equation}
 Thus, the net reproductive rate is
$$
\mathcal{R}_{0}=r(-\mathcal{F}\mathcal{B}^{-1})=r(R_{1}).
$$
Since the operator $(m+s-d\Delta)^{-1}$ is compact and positive (if $h\ge(\not\equiv) 0_Y\in Y$, then  $(m+s-d\Delta)^{-1}h\gg 0_Y$), $R_1$ is compact and positive. By the Krein-Rutman theorem, $\mathcal{R}_0=r(R_1)$ is  an eigenvalue of  $R_1$ that corresponds with a nonnegative eigenfunction $v\in Y$. So $v$ satisfies $R_1 v =\mathcal{R}_0v$. Denote $\phi=(m+s-d\Delta)^{-1}v$. Then $\phi$ is positive and an eigenfunction of \eqref{Reig} corresponding with eigenvalue $\lambda=1/\mathcal{R}_0$
if $\mathcal{R}_0>0$.
\end{proof}

\begin{remark}
    We note that $\mathcal{R}_0$ is defined to be the spectral radius of $R_1$ in \cite{Deng2019}, where $R_1$ is given in \eqref{R1}. Thus, the two definitions of $\mathcal{R}_0$  in  \cite{Deng2019} and the current paper are equivalent.
\end{remark}


\subsection{Growth bound}

In this subsection, we consider the exponential growth bound of the linearized system \eqref{linearized model} and study its relation  with the reproduction rate $\mathcal{R}_0$ and the eigenvalues of its generator. We write \eqref{linearized model} as:
\begin{equation}
  \left\{
   \begin{aligned}
&\  \frac{d\bm x}{dt}=\mathcal{L}\bm x, \quad \forall t\geq 0, &\\
&\  \bm x(0)\in \overline{D(\mathcal{L})}=\mathbb{X}_0,
   \end{aligned}
  \right.
  \end{equation}
where $\mathcal{L}=A+D$ and
$D: \mathbb{X}_{0}\rightarrow \mathbb{X}$ denotes the derivative of $F$ at $E_{0}$, i.e.,
\begin{equation*}
D\begin{gathered}\begin{pmatrix}
u\\ \begin{pmatrix}
0_Y\\ w
\end{pmatrix}
\end{pmatrix}
=\begin{pmatrix}
\int_{0} ^{a_{\max}}\beta(\cdot, a,0) w(\cdot, a)da-(m+e)u\\ \begin{pmatrix}
\chi(\cdot, 0)e u\\
0_Z
\end{pmatrix}
\end{pmatrix},
\quad (u, (0_Y, w)) \in X_0.
\end{gathered}
\end{equation*}
Here, $A$ is defined by \eqref{A} with $\mu_0=\mu(\cdot, \cdot, 0)$.
By Lemma \ref{H-Y}, we know that $A: D(A)\subset \mathbb{X}\to \mathbb{X}$ is a Hille-Yosida operator. Since $\mathcal{L}$ is a bounded perturbation of $A$, it is also a Hille-Yosida operator (\cite[Theorem 3.5.5]{arendtvector}). Hence, $\mathcal{L}_0$, the part of $\mathcal{L}$ on $\mathbb{X}_0$, is the infinitesimal generator of a strongly continuous semigroup $\{T_{\mathcal{L}_0}(t)\}_{t\geq 0}$ on $\overline{D(\mathcal{L}_0)}=\mathbb{X}_0$.
The \emph{exponential growth bound} of  $\{T_{\mathcal{L}_0}(t)\}_{t\geq 0}$ is defined as
$$
\omega=\omega(\mathcal{L}_0):=\lim_{t\rightarrow \infty}\frac{\ln (\|T_{\mathcal{L}_0}(t)\|)}{t}.
$$
 Then, for any $\omega'>\omega$, there is a positive constant $C$ such that  $\|T_{\mathcal{L}_0}(t)\|\leq C e^{\omega' t}$ for all $t\ge 0$.
 It is well-known (see \cite{Engel2000}) that
\begin{equation}\label{L0}
 \omega(\mathcal{L}_0)=\max\{s(\mathcal{L}_0),\omega_{ess}(\mathcal{L}_0)\},
 \end{equation}
 where $s(\mathcal{L}_0)$ is the spectral bound of $\mathcal{L}_0$ and  $\omega_{ess}(\mathcal{L}_0)$ is the essential growth bound of $\mathcal{L}$ defined by
 $$
 \omega_{ess}(\mathcal{L}_0):=\lim_{t\rightarrow \infty}\frac{\ln(\alpha(T_{\mathcal{L}_0}(t)))}{t}.
 $$
Here $\alpha$ is the  measure of non-compactness, that is, for any bounded linear operator $\mathcal{H}$ on $\mathbb{X}_0$,
$$\alpha(\mathcal{H}):=\inf_{S\in \mathcal{S}}\|\mathcal{H}-\mathcal{S}\|,$$
where $\mathcal{S}$ is the subset of compact linear operators on $\mathbb{X}_0$.

Our next result shows that $\omega_{ess}(\mathcal{L}_0)$ is negative.
\begin{lemma}\label{lemma_w}
    If $a_{\max}=\infty$, then $\omega_{ess}(\mathcal{L}_0)\le -\underline\mu$; if $a_{\max}<\infty$, then $\omega_{ess}(\mathcal{L}_0)= -\infty$.
\end{lemma}
\begin{proof}
We write $\mathcal{L}=\mathcal{E}+\mathcal{C}$, where
\begin{equation*}
\begin{gathered}\mathcal{E}\bm x
=\begin{pmatrix}
d\Delta u-(m+e)u+\int_{0} ^{a_{\max}}\beta(\cdot,a,0) w(\cdot, a) da \\
-w(\cdot,0)\\
-w_{a}-\mu(\cdot,\cdot,0) w
\end{pmatrix}, \quad\forall \bm x=(u, (0_Y, w)) \in D(\mathcal{E}),
\end{gathered}
\end{equation*}
and
\begin{equation*}
\begin{gathered}\mathcal{C}\bm x
=\begin{pmatrix}
0_Y \\
\chi(\cdot,0)e u\\
0_Z
\end{pmatrix}, \quad\forall \bm x=(u, (0_Y, w)) \in\mathbb{X}_0.
\end{gathered}
\end{equation*}
Then, $\mathcal{E}$ is a bounded linear perturbation of $A$, so it is a  Hille-Yosida operator. So, $\mathcal{E}_0$, the part of $\mathcal{E}$ on $\mathbb{X}_0$, generates a strongly continuous semigroup $\{T_{\mathcal{E}_0}(t)\}_{t\ge 0}$ on $\overline{D(\mathcal{E})}=\mathbb{X}_0$.

Let $B=\{\bm x_0\in \mathbb{X}_0: \ \|\bm x_0\|_{\mathbb{X}}\le r \}$ for some $r>0$.
Let $\bm x_0=(u_0, (0_Y, w_0)) \in B$ and $(u(\cdot, t), (0_Y, w(\cdot, \cdot, t)))=T_{\mathcal{E}_0}(t)\bm x_0$, i.e., the solution of
  \begin{equation}\label{Esol}
  \left\{
   \begin{array}{lll}
\partial_{t} u=d\Delta u +\displaystyle\int_{0} ^{a_{\max}}\beta(x,a,0)w(x,a,t)da-(m(x)+e(x))u, \quad &x\in\Omega, t>0, \\
\partial_{t}w+\partial_{a}w=-\mu(x,a, 0)w ,\quad &x\in\Omega, t>0,a\in (0,a_{\max}), \\
\partial_\nu u=0,~\quad &x\in \partial \Omega, t>0, \\
 w(x,0,t)=0_Y,  \quad &x\in\Omega, t>0, \\
 u(x,0)=u_{0}(x), \quad &x\in\Omega, \\
 w(x,a,0)=w_{0}(x,a),\quad  &x\in\Omega, a\in (0, a_{\max}).
   \end{array}
  \right.
  \end{equation}
It is easy to see that the equations of $w$ decouples from system \eqref{Esol}. So we can solve for $w$ to obtain
\begin{equation}\label{Esol1}
    w(x, a, t)=
      \begin{cases}
    w_0(x, a-t) e^{-\int_{a-t}^a \mu(x, s, 0)ds}, \ &\text{if}\ a> t, \\
0_Y, \ &\text{if}\ a<t.
\end{cases}
\end{equation}
Suppose $a_{\max}=\infty$. Then, for all $t\ge 0$,
\begin{eqnarray}
\|w(\cdot, \cdot, t)\|_Z&=&\int_0^\infty \|w(\cdot, a, t)\|_\infty da \nonumber\\
&=&\int_t^\infty \|w_0(\cdot, a-t) e^{-\int_{a-t}^a \mu(\cdot, s, 0)ds} \|_\infty da \nonumber\\
&\le&e^{-\underline\mu t} \|w_0\|_Z\le re^{-\underline\mu t}. \nonumber
\end{eqnarray}
Notice that
$$
\| \int_{0} ^{a_{\max}}\beta(\cdot,a,0)w(\cdot,a,t)da \|_\infty\le \bar\beta \|w(\cdot, \cdot, t)\|_Z\le \bar\beta r e^{-\underline\mu t}.
$$
By the well-known parabolic estimates and embedding theorems, the set $\{u(\cdot, t): \ \bm x_0\in B\}$ is precompact in $Y$. Therefore, we have
\begin{eqnarray*}
\alpha(T_{\mathcal{E}_0}(t)B)&\leq& \alpha(\{u(\cdot,\cdot, t): \ \bm x_0\in B\})+\alpha(\{w(\cdot,\cdot, t): \ \bm x_0\in B\}\\
&=&0+\alpha(\{w(\cdot,\cdot, t): \ \bm x_0\in B\}\\
&\le & \|\{w(\cdot, t): \ \bm x_0\in B\}\|_Z\le re^{-\underline\mu t}.
\end{eqnarray*}
This implies that $\omega_{ess}(\mathcal{E}_0)\leq -\underline{\mu}$.

If $a_{\max}<\infty$, by \eqref{Esol1}, we have $\|w(\cdot, \cdot, t)\|_Z=0$ for all $t\ge a_{\max}$. Repeating the above computations, we will get $\alpha(T_{\mathcal{E}_0}(t)B)=0$ and $\omega_{ess}(\mathcal{E}_0)= -\infty$ (in this case, $T_{\mathcal{E}_0}(t)$ is compact for all $t> a_{\max}$).

Finally, by the definition of $\mathcal{C}$ and the compactness of $\{u(\cdot, t): \ \bm x_0\in B\}$, the operator $\mathcal{C} T_{\mathcal{E}_0}(t)$, $t>0$, is compact on $\mathbb{X}_0$.
 It follows from \cite[Theorem 1.2]{Ducrot2008} that
$$
\omega_{ess}(\mathcal{L}_{0})\le\omega_{ess}(\mathcal{E}_0)
\begin{cases}
\leq -\underline{\mu}<0, & \text{if}\ a_{\max}=\infty,\\
=-\infty, & \text{if}\ a_{\max}<\infty.
\end{cases}
$$
\end{proof}

It is clear that $\rho(\mathcal{L})\neq\emptyset$, so $\rho(\mathcal{L})=\rho(\mathcal{L}_0)$ (\cite[Lemma 2.2.10]{Ruan2018}) and $s(\mathcal{L})=s(\mathcal{L}_0)$.
The following is the main result in this section, which relates the reproductive rate $\mathcal{R}_0$, the growth bound $\omega(\mathcal{L}_0)$, and the spectral bound $s(\mathcal{L}_0)$.

\begin{theorem}\label{theorem_R}
    The following statements are valid:
\begin{enumerate}
    \item[{\rm (i)}] Suppose $a_{\max}<\infty$. Then, $\mathcal{R}_0-1$ has the same sign as $\omega(\mathcal{L}_0)$. Moreover,  $\omega(\mathcal{L}_0)=s(\mathcal{L}_0)$ is the principal eigenvalue of $\mathcal{L}_0$ in the sense that $\lambda_0:=\omega(\mathcal{L}_0)=s(\mathcal{L}_0)$ is a simple eigenvalue of the following problem
      \begin{equation}\label{eigenvalue problem:2}
  \left\{
   \begin{array}{lll}
\lambda\phi=d\Delta \phi +\displaystyle\int_{0} ^{a_{\max}}\beta(x,a,0)\varphi(x,a)da-m(x)\phi-e(x)\phi, \quad &x\in\Omega, \\
\lambda\varphi+ \partial_{a}\varphi=-\mu(x,a,0)\varphi, \quad &x\in\Omega, a\in (0,a_{\max}), \\
\partial_\nu \phi=0,~\quad &x\in \partial \Omega,  \\
 \varphi(x,0)= \chi(x,0)e(x)\phi,  \quad &x\in\Omega, \\
   \end{array}
  \right.
  \end{equation}
  which is the unique eigenvalue that corresponds with a positive eigenfunction, and any other spectrum value $\lambda$ of $\mathcal{L}_0$ satisfies $Re(\lambda)<\lambda_0$.

    \item[{\rm (ii)}] Suppose $a_{\max}=\infty$. If $\mathcal{R}_0<1$, then $\omega(\mathcal{L}_0)<0$; if $\mathcal{R}_0=1$, then $\omega(\mathcal{L}_0)=s(\mathcal{L}_0)=0$; if $\mathcal{R}_0>1$, then $\omega(\mathcal{L}_0)=s(\mathcal{L}_0)>0$. Moreover, if  $\mathcal{R}_0\ge1$, $\lambda_0:=\omega(\mathcal{L}_0)=s(\mathcal{L}_0)$ is the principal eigenvalue of $\mathcal{L}_0$ (in the sense as {\rm (i)}).
\end{enumerate}
\end{theorem}
\begin{proof}
The relation of $\mathcal{R}_0$, $s(\mathcal{L})$, and $\omega(\mathcal{L}_0)$ follow from Proposition \ref{prop1}, \eqref{L0} and Lemma \ref{lemma_w}. It remains to show that $\lambda_0=\omega(\mathcal{L}_0)=s(\mathcal{L})$ is the principal eigenvalue of $\mathcal{L}_0$ when $a_{\max}<\infty$ or $a_{\max}=\infty$ and $\mathcal{R}_0\ge 1$.
By \cite[Corollary IV 2.11]{Engel2000} (also see \cite{Webb1985}), the set
$$
\{\lambda:\  \lambda\in \sigma(\mathcal{L}_0)\quad \text{and}\quad Re(\lambda)> \omega_{ess}(\mathcal{L}_0)\}
$$
is finite and composed of isolated eigenvalues of $\mathcal{L}_0$ with finite algebraic multiplicity. 

Suppose  that either $a_{\max}<\infty$ or $a_{\max}=\infty$ and $\mathcal{R}_0\ge 1$.
By Lemma \ref{lemmaL}, $\mathcal{L}_0$ is  resolvent positive, which means that the semigroup $\{T_{\mathcal{L}_0}(t)\}_{t\ge 0}$ on $\mathbb{X}_0$ generated by $\mathcal{L}_0$ is positive \cite[Theorem VI.1.8]{Engel2000}. By \cite[Proposition 2.5]{webb1987operator}, Lemma \ref{lemma_w} and \eqref{L0}, $\lambda_0=\omega(\mathcal{L}_0)=s(\mathcal{L})$ is an eigenvalue of the operator $\mathcal{L}_0$, $\lambda_0$ associates with a positive eigenfunction, and it is the unique peripheral eigenvalue in the sense that any other eigenvalue $\lambda\neq \lambda_0$ of $\mathcal{L}_0$ satisfies $Re(\lambda)<\lambda_0$.

We only need to show that eigenvalue $\lambda_0$ of $\mathcal{L}_0$ is simple and it is the unique eigenvalue that corresponds with a positive eigenfunction. It is easy to see that $(\lambda, \phi, \varphi)$ is an eigenpair of \eqref{eigenvalue problem:2} (i.e. eigenpair of $\mathcal{L}_0$) if and only if $(\lambda, \phi)$ is an eigenpair of the following nonlinear eigenvalue problem (if $a_{\max}=\infty$, assume $\lambda>-\underline\mu$):
   \begin{equation}\label{eigenvalue problem:3}
  \left\{
   \begin{array}{lll}
\lambda \phi=d\Delta \phi + \chi(x,0)e(x)\displaystyle\int_{0} ^{a_{\max}}\beta(x,a,0)e^{-\int_{0}^{a}\mu(x,l,0)dl}e^{-\lambda a}da\phi-(m(x)+e(x))\phi, \quad &x\in\Omega, \\
\partial_\nu \phi=0,~\quad &x\in \partial \Omega.  \\
   \end{array}
  \right.
  \end{equation}
Indeed, if $(\lambda, \phi, \varphi)$ is an eigenpair of
 \eqref{eigenvalue problem:2}, by the equation of $w$,
\begin{equation}\label{def_phi}
\varphi(x,a)=\chi(x,0) e(x) e^{-\int_{0}^{a}\mu(x,l,0)dl}e^{-\lambda a}\phi(x), \quad a\in (0, a_{\max}), x\in\Omega.
\end{equation}
Substituting it into the first equation of \eqref{eigenvalue problem:2}, we can see that $(\lambda, \phi)$ is an eigenpair of
\eqref{eigenvalue problem:3}. On the contrary, suppose that $(\lambda, \varphi)$ is an eigenpair of \eqref{eigenvalue problem:3}. Defining $\varphi$ by \eqref{def_phi}, it is easy to check that $(\lambda, \phi, \varphi)$ is an eigenpair of \eqref{eigenvalue problem:2}. Note that if $a_{\max}=\infty$,  the assumption $\lambda>-\underline\mu$ guarantees that $\varphi\in Z$ and the integral in the first equation of \eqref{eigenvalue problem:3} converges.

Now, we know that $\lambda_0$ is an eigenvalue of \eqref{eigenvalue problem:3}. To see the properties of $\lambda_0$, for each $k\in\mathbb{R}$ (if $a_{\max}=\infty$, assume $k>-\underline\mu$), consider the eigenvalue problem:
   \begin{equation}\label{eigenvalue problem:k}
  \left\{
   \begin{array}{lll}
\lambda \phi=d\Delta \phi + \chi(x,0)e(x)\displaystyle\int_{0} ^{a_{\max}}\beta(x,a,0)e^{-\int_{0}^{a}\mu(x,l,0)dl}e^{-k a}da\phi-(m(x)+e(x))\phi, \quad &x\in\Omega, \\
\partial_\nu \phi=0,~\quad &x\in \partial \Omega.  \\
   \end{array}
  \right.
  \end{equation}
It is well-known that \eqref{eigenvalue problem:k} has a principal eigenvalue
$\hat\lambda_0(k)$ (it is a simple and real eigenvalue, and it is also the unique eigenvalue that corresponds with a positive eigenvector).
Since $\lambda_0$ associates with a positive eigenfunction, $k=\lambda_0$ is a root of
the following equation
\begin{equation}\label{eqk}
\hat\lambda_0(k)=k.
\end{equation}
Actually, since $e^{-ka}$ is decreasing in $k$, by the variational characterization of the eigenvalue $\hat\lambda_{0}(k)$,   $\hat\lambda_{0}(k)$ is decreasing in $k$.
If $a_{\max}<\infty$, \eqref{eqk} has a unique root in $\mathbb{R}$. If   $a_{\max}=\infty$ and $\mathcal{R}_0\ge 1$, $1/\mathcal{R}_0\le 1$ is the principal  eigenvalue of \eqref{Reig}. This implies $\hat\lambda_0(0)\ge 0$. Since  $\hat\lambda_0(k)$ is well-defined and decreasing in $k\in [0, \infty)$, \eqref{eqk} has a unique root in $[0, \infty)$.

Since the eigenvalue $\hat\lambda_{0}(k)$ of \eqref{eigenvalue problem:k} is simple, the geometric multiplicity  of $\lambda_0$ for \eqref{eigenvalue problem:3} (hence \eqref{eigenvalue problem:2}) is 1, i.e. the dimension of the null space $ker(\lambda_0 I-\mathcal{L}_0)$ is 1. Next, we show that the algebraic multiplicity of $\lambda_0$ is 1. It suffices to show that  $ker(\lambda_0 I-\mathcal{L}_0)^2\subset ker(\lambda_0 I-\mathcal{L}_0)$. To this end, let $\bm\psi=(\phi, (0_Y, \varphi)) $ such that
$$
(\lambda_0 I-\mathcal{L}_0)\bm\psi=\bm\psi_0=(\phi_0, 0_Y, \varphi_0) \in ker(\lambda_0I-\mathcal{L}_0).
$$
 The above equation can be written as
\begin{equation}\label{eigS}
  \left\{
   \begin{array}{lll}
\lambda_0\phi-d\Delta \phi -\displaystyle\int_{0} ^{a_{\max}}\beta(x,a,0)\varphi(x,a)da+(m(x)+e(x))\phi=\phi_0, \quad &x\in\Omega, \\
\lambda_0\varphi+ \partial_{a}\varphi+\mu(x,a,0)\varphi=\varphi_0, \quad &x\in\Omega, a\in (0,a_{\max}), \\
\partial_\nu \phi=0,~\quad &x\in \partial \Omega,  \\
 \varphi(x,0)= \chi(x,0)e(x)\phi,  \quad &x\in\Omega. \\
   \end{array}
  \right.
  \end{equation}
Solving the second and fourth equation of \eqref{eigS} and using \eqref{def_phi}, we obtain
\begin{eqnarray*}
 \varphi&=&e^{-\int_0^a (\mu(x, s, 0)+\lambda_0)ds} \chi(x, 0) e \phi + \int_0^a \varphi_0(x, l) e^{-\int_l^a(\mu(s)+\lambda_0)ds}dl\\
 &=&e^{-\int_0^a (\mu(x, s, 0)+\lambda_0)ds} \chi(x, 0) e \phi + \int_0^a  e^{-\int_0^a(\mu(s)+\lambda_0)ds}dl\, \chi(x, 0)e  \phi_0.
\end{eqnarray*}
Substitute this into the first equation of \eqref{eigS}, we have
\begin{eqnarray}\label{ct}
&&\lambda_0\phi-d\Delta \phi -\displaystyle\int_{0} ^{a_{\max}} \beta(x, a, 0)e^{-\int_0^a(\mu(x, s, 0)+\lambda_0)ds} da\, \chi e \phi+(m+e)\phi \nonumber\\
&&=\int_0^{a_{\max}}\beta(x, a, 0) a e^{-\int_0^a(\mu(x, s, 0)+\lambda_0)ds}   da\, \chi(x, 0)e \phi_0+\phi_0.
\end{eqnarray}
Note that $\phi_0$ satisfies
\begin{equation}\label{phi00}
  \lambda_0\phi_0-d\Delta \phi_0 -\displaystyle\int_{0} ^{a_{\max}} \beta(x, a, 0)e^{-\int_0^a(\mu(x, s, 0)+\lambda_0)ds} da\, \chi(x,0) e \phi_0+(m+e)\phi_0 =0.
\end{equation}
Multiplying both sides of \eqref{ct} by $\phi_0$, integrating over $\Omega$,  integrating by parts and using \eqref{phi00},
we obtain
$$
0=\int_\Omega \left( \int_0^{a_{\max}}\beta(x, a, 0) a e^{-\int_0^a(\mu(x, s, 0)+\lambda_0)ds}   da\, \chi(x, 0) e(x)+1   \right) \phi_0^2 dx,
$$
which means $\phi_0=0_Y$ and $\varphi_0=0_Z$. Therefore, $\bm\psi\in ker(\lambda_0I-\mathcal{L}_0)$  and  $ker(\lambda_0 I-\mathcal{L}_0)^2\subset ker(\lambda_0 I-\mathcal{L}_0)$. So, the algebraic multiplicity of $\lambda_0$ is 1.
\end{proof}

From the proof of Theorem \ref{theorem_R}, we can obtain the following result, which further relates the sign of $\mathcal{R}_0-1$ and the principal eigenvalue of an elliptic problem.
\begin{proposition}\label{prop_eig0}
    The following statements are valid:
    \begin{enumerate}
    \item[{\rm (i)}] If $a_{\max}=\infty$, suppose that $\mathcal{R}_0\ge 1$. Then $s(\mathcal{L}_0)$ is the principal eigenvalue of \eqref{eigenvalue problem:3};
    \item[{\rm (ii)}] $\mathcal{R}_0-1$ has the same sign as the principal eigenvalue of
       \begin{equation}\label{eigenvalue problem:0}
  \left\{
   \begin{array}{lll}
\lambda \phi=d\Delta \phi + \chi(x,0)e\displaystyle\int_{0} ^{a_{\max}}\beta(x,a,0)e^{-\int_{0}^{a}\mu(x,l,0)dl}da\phi-(m+e)\phi, \quad &x\in\Omega, \\
\partial_\nu \phi=0,~\quad &x\in \partial \Omega.  \\
   \end{array}
  \right.
  \end{equation}
 \end{enumerate}
\end{proposition}
\begin{proof}
(i) This has been shown in the proof of Theorem \ref{theorem_R}.

(ii) Let $\hat \lambda_0(k)$ be defined as in the proof of Theorem \ref{theorem_R}.
Suppose  that either $a_{\max}<\infty$ or $a_{\max}=\infty$ and $\mathcal{R}_0\ge 1$.
We have already known that $k=\lambda_0:=s(\mathcal{L}_0)$ is the unique root of equation \eqref{eqk}. Clearly, the root of \eqref{eqk} has the same sign as $\hat\lambda_0(0)$, which is the principal eigenvalue of \eqref{eigenvalue problem:0}. Finally, suppose that $\mathcal{R}_0< 1$ and $a_{\max}=\infty$. Suppose to the contrary that $\hat\lambda_0(0)\ge 0$. Then equation \eqref{eqk} has a unique solution $\bar s\ge 0$. Then $\bar s$ will be an eigenvalue of \eqref{eigenvalue problem:3}, and hence an eigenvalue of \eqref{eigenvalue problem:2}. This implies that $s(\mathcal{L}_0)\ge \bar s\ge 0$, which contradicts that the fact that $\mathcal{R}_0-1$ has the same sign as $s(\mathcal{L}_0)$ (Proposition \ref{prop1}). Therefore, $\hat\lambda_0(0)< 0$, which has the same sign as $\mathcal{R}_0-1$.
\end{proof}

\begin{remark}
    The eigenvalues of \eqref{eigenvalue problem:3}  are related to the stability of the trivial solution of the following reaction-diffusion with time delay:
    \begin{equation*}\label{delay}
  \left\{
   \begin{array}{lll}
u_t=d\Delta u + \chi(x,0)e(x)\displaystyle\int_{0} ^{a_{\max}}\beta(x,a,0)e^{-\int_{0}^{a}\mu(x,l,0)dl} u(x, t-a)da\phi-(m(x)+e(x))u, \quad &x\in\Omega, \\
\partial_\nu u=0,~\quad &x\in \partial \Omega.  \\
   \end{array}
  \right.
  \end{equation*}
 We refer the interested readers to \cite{wu2012theory} for this topic. Moreover, if $a_{\max}<\infty$ (the time delay is finite), the relation of the signs of the principal eigenvalues of problems \eqref{eigenvalue problem:3} and \eqref{eigenvalue problem:0} can be found in many references (e.g., \cite[Theore 2.2]{thieme2001non} and \cite[Section 4]{kerscher1984asymptotic}).
\end{remark}

\section{Global stability}\label{Stability of the trivial steady state}

In this section, we study the global stability of the equilibria of model \eqref{model}. We will show that $\mathcal{R}_0=1$ is a threshold value for the global dynamics of the model.

\subsection{Global stability of trivial equilibrium when $\mathcal{R}_0<1$}
In this subsection, we  investigate the global stability of the trivial equilibrium $E_{0}$ if $\mathcal{R}_0<1$.

We need the following hypotheses:
\begin{itemize}
\item[(A6)] For each $(x, a)\in \Omega\times (0, a_{\max})$, the function $[0, \infty)\ni P\to \beta(x, a, P)$ is decreasing in $P$, the function $[0, \infty)\ni P\to \mu(x, a, P)$ is increasing in $P$, and the  function $[0, \infty)\ni P\to \chi(x, P)$ is decreasing in $P$.
\end{itemize}

The following result states that the trivial equilibrium is globally stable if $\mathcal{R}_0<1$.
\begin{theorem}\label{theorem_Rsmall}
Suppose that $(A1)$-$(A3)$ and $(A5)$-$(A6)$ hold. If $\mathcal{R}_0<1$, then the trivial equilibrium $E_0$ of \eqref{model} is globally attractive.
\end{theorem}
\begin{proof}
Let $\bm x_0=(u_0, (0_Y, w_0)) \in \mathbb{X}_{0+}$ and $U(t)\bm x_0$ be the integrated solution of \eqref{model}. Let $A$ be defined by \eqref{A} with $\mu_0=0$. Then by (A6),  $U(t)\bm x_0$ satisfies
\begin{eqnarray*}
U(t)\bm x_0&=&T_{(A-\gamma Q)_0}(t)\bm x_0+S_{A-\gamma Q}\diamond (\gamma Q+F)(U(\cdot+s)\bm x_0)(t-s)\\
&\le& T_{(A-\gamma Q)_0}(t)\bm x_0+S_{A-\gamma Q}\diamond (\gamma Q+F_0)(U(\cdot+s)\bm x_0)(t-s), \quad \forall t\ge 0,
\end{eqnarray*}
where
\begin{equation*}
F_0\begin{gathered}\begin{pmatrix}
u\\ \begin{pmatrix}
0_Y\\ w
\end{pmatrix}
\end{pmatrix}
=\begin{pmatrix}
\int_{0} ^{a_{\max}}\beta(\cdot, a, 0) w(\cdot, a)da-(m+e)u\\ \begin{pmatrix}
\chi(\cdot, 0)e u\\-\mu(\cdot,\cdot, 0) w
\end{pmatrix}
\end{pmatrix}.
\end{gathered}
\end{equation*}
By Proposition \ref{supper solution: integral form}, $U(t)\bm x_0\le V(t)\bm x_0$ for all $t\ge 0$, where $V(t)\bm x_0$ is the integrated solution of
\begin{equation}\label{model_low}
  \left\{
   \begin{array}{lll}
\partial_{t} u = d\Delta u +\int_0^{a_{\max}}\beta(x, a, 0)w(x, a, t)da-(m(x)+e(x))u, \quad &x\in\Omega, t>0, \\
\partial_{t}w+\partial_{a}w= -\mu(x,a, 0)w ,\quad &x\in\Omega, t>0,a\in (0,a_{\max}), \\
\partial_\nu u=0,~\quad &x\in \partial \Omega, t>0, \\
 w(x,0,t)= \chi(x,0)e(x)u(x,t),  \quad &x\in\Omega, t>0, \\
 u(x,0)=u_{0}(x), \quad &x\in\Omega, \\
 w(x,a,0)=w_{0}(x,a),\quad  &x\in\Omega, a\in (0, a_{\max}).
   \end{array}
  \right.
  \end{equation}
 Clearly, the solution of \eqref{model_low} is $V(t)\bm x_0=T_{\mathcal{L}_0}(t)\bm x_0$. By Theorem \ref{theorem_R} and  $\mathcal{R}_0<1$, we have $\omega(\mathcal{L}_0)<0$. Hence, there exist $\omega'<0$ and $C>0$ such that $\|T_{\mathcal{L}_0}(t)\|\le Ce^{\omega' t}$ for all $t\ge 0$. It follows that $\lim_{t\to\infty} V(t)\bm x_0\to 0_{\mathbb{X}}$ in $\mathbb{X}$ as $t\to\infty$.  So, $E_0$ is globally attractive.
\end{proof}

\begin{remark}
    In Theorem \ref{theorem_Rsmall}, we do not require (A4). In the proof, when we apply  Proposition \ref{supper solution: integral form}, we use that $U(t)\bm x_0$ is a lower solution of \eqref{model_low}. Clearly, the coefficients $\beta(x, a, 0)$, $\mu(x, a, 0)$, and $\chi(x, 0)$ of  \eqref{model_low} satisfy the assumptions in (A4).
\end{remark}

\subsection{Uniform persistence when $\mathcal{R}_0>1$}


For each $(x, a)\in \Omega\times (0, a_{\max})$, let $\beta_\infty(x, a):=\inf_{P\ge 0} \beta(x, a, P)$. Let
{\small
$$
\partial\mathbb{X}^0_0=\left\{(u_0, (0_Y, w_0)) \in\mathbb{X}_0: \ u_0\equiv 0\ \text{and}\  \int_t^{a_{\max}} \int_\Omega \beta_\infty(x, a)w_0(x, a-t)dxda=0 \ \text{for all} \ t\in [0, a_{\max})\right\}
$$
}
and
{\small
$$
 \mathbb{X}_0^0=\left\{(u_0, (0_Y, w_0)) \in\mathbb{X}_0: \ u_0\not\equiv 0\ \text{or}\  \int_t^{a_{\max}} \int_\Omega \beta_\infty(x, a)w_0(x, a-t)dxda>0 \ \text{for some} \ t\in [0, a_{\max})\right\}.
$$
}
Then, $\mathbb{X}_0=\partial \mathbb{X}_0^0\cup \mathbb{X}^0_0$.
\begin{theorem}\label{theorem_persist}
Suppose that $(A1)$-$(A3)$ and $(A5)$-$(A6)$ hold. If $\mathcal{R}_0>1$, then \eqref{model} is uniformly persistent in the sense that there exists $\epsilon_0>0$ such that the solution of \eqref{model} satisfies
\begin{equation}\label{persist}
\liminf_{t\to\infty} \min_{x\in\bar\Omega} u(x, t)\ge \epsilon_0 \quad\text{and}\quad \liminf_{t\to\infty} \int_0^{a_{\max}}\int_\Omega w(x, a, t)dxda\ge \epsilon_0
\end{equation}
for any initial data in $\mathbb{X}_0^0$.
\end{theorem}
\begin{proof}
We divide the proof into several steps.
\\
\emph{Step 1. Claim: For any $\bm x_0=(u_0, (0_Y, w_0)) \in \mathbb{X}_0^0$, there exists $t_0>0$ such that the solution of \eqref{model} satisfies $u(x, t)>0$ and $w(x, a, t)>0$ for all $x\in\bar\Omega$, $t>t_0$ and $a<t$.}


Let $U(t)\bm x_0=(u(\cdot, t), (0_Y, w(\cdot, \cdot, t)))$ be the solution of \eqref{model}. If $u_0\not\equiv 0$, by the maximum principle of parabolic equations, we have $u(x, t)>0$ for all $x\in\bar\Omega$ and $t>0$. Then by \eqref{wint}, we know $w(x, a, t)>0$ for all $x\in\bar\Omega$, $t>0$ and $a<t$.

If $w_0$ satisfies
\begin{equation}\label{w0}
\int_{t_0}^{a_{\max}} \int_\Omega \beta_\infty(x, a)w_0(x, a-t_0)dxda>0
\end{equation}
for some $t_0\in [0, a_{\max})$, by \eqref{wint}, we have
\begin{eqnarray*}
&&\int_0^{a_{\max}} \int_\Omega \beta(x, a, P)w(x, a, t_0)dxda\\
&&\ge  \int_{t_0}^{a_{\max}}\int_\Omega\beta_\infty (x, a) w_0(x, a-{t_0}) e^{-\int_0^{t_0} \mu(x, a-s, P(t_0-s))ds} dx da>0.
\end{eqnarray*}
So by the first equation of \eqref{model} and parabolic maximum principle, we have $u(x, t)>0$ for all $t>t_0$ and $x\in\bar\Omega$.  Then by \eqref{wint}, we know $w(x, a, t)>0$ for all $x\in\bar\Omega$, $t>t_0$ and $a<t$. This proves the claim.

\vspace{0.5cm}
    Define
$$
\rho(\bm x)=\min\left\{\min_{x\in\bar\Omega} u(x), \ \int_0^{a_{\max}} \int_\Omega w(x, a)dxda \right\}, \quad\forall \bm x=(u, (0_Y, w)) \in \mathbb{X}_0.
$$
\\
\emph{Step 2. Claim: $\rho$ is a generalized distance function in the sense that for any $\bm x_0\in (\mathbb{X}_0^0\cap \rho^{-1}(0))\cup \rho^{-1}(0, \infty)$, we have $\rho(U(t)\bm x_0)>0$}
for large $t$.

Let $\bm x_0=(u_0, (0_Y, w_0)) \in \rho^{-1}(0, \infty)$. Then, $u_0\not\equiv 0$. By the proof of step 1,  we have $u(x, t)>0$ and $w(x, a, t)>0$ for all $x\in\bar\Omega$, $t>0$ and $a<t$. Hence, $\rho(U(t)\bm x_0)>0$ for all $t>0$.

Let $\bm x_0\in \mathbb{X}_0^0\cap \rho^{-1}(0)$. Then either $u_0\not\equiv 0$ or $w_0$ satisfies \eqref{w0}. It suffices to  consider the case  that $w_0$ satisfies \eqref{w0}. By Step 1, there exists $t_0\in (0, a_{\max})$ such that $u(x, t)>0$ and $w(x, a, t)>0$ for all $x\in\bar\Omega$, $t>t_0$ and $a<t$. Hence, $\rho(U(t)\bm x_0)>0$ for all $t>t_0$. This proves the claim.

\vspace{0.5cm}
Define $M_\partial:=\{\bm x_0\in \partial \mathbb{X}_0^0: \ U(t)\bm x_0\in\partial\mathbb{X}_0^0,\quad \forall\ t\ge 0\}$.
\\
\emph{Step 3. Claim: The $\omega$-limit set $\omega(\bm x_0)$ is $\{E_0\}$ for any $\bm x_0\in M_\partial$, where $E_0$ is the trivial equilibrium.}

Let $\bm x_0=(u_0, (0_Y, w_0)) \in M_\partial$. By the definition of $M_\partial$ and $\partial \mathbb{X}_0^0$, we have $u(x, t)= 0$ for all $x\in\bar\Omega$ and $t\ge 0$. By \eqref{wint},
\begin{equation*}
    w(x, a, t)=
      \begin{cases}
    w_0(x, a-t) e^{-\int_0^t \mu(x, a-s, P(t-s))ds}, \ &\text{if}\ a> t, \\
0, \ &\text{if}\ a<t,
\end{cases}
\end{equation*}
for all $x\in\bar\Omega$. It follows that
\begin{eqnarray*}
 \int_0^{a_{\max}}\|w(\cdot, a, t)\|_\infty da&=&\int_t^{a_{\max}} \|w_0(\cdot, a-t) e^{-\int_0^t \mu(\cdot, a-s, P(t-s))ds} \|_\infty da  \\
 &\le & \int_t^{a_{\max}} \|w_0(\cdot, a-t) e^{-\underline\mu t}\|_\infty da\\
 &= & e^{-\underline\mu t}  \|w_0\|_Z\to 0\quad \text{as}\ t\to\infty.
\end{eqnarray*}
Therefore, we must have $\omega(\bm x_0)=\{E_0\}$.

\vspace{0.5cm}
\noindent\emph{Step 4. Claim: $W^s(E_0)\cap \rho^{-1}(0, \infty)=\emptyset$, where $W^s(E_0)$ is the stable set of $E_0$.  }

Suppose to the contrary that there exists $\bm x_0=(u_0, (0_Y, w_0))  \in W^s(E_0)\cap \rho^{-1}(0, \infty)$. Choose  $\delta_0>0$ small, which will be specified later. Since $U(t)\bm x_0\to E_0$ in $\mathbb{X}$ as $t\to\infty$, there exists $t'>0$ such that $\|u(\cdot, t)\|_\infty\le \delta_0$ and $P(t)=\int_0^{a_{\max}}\int_\Omega w(x, a, t)dxda\le \delta_0$
for all $t\ge t'$.

 Recall that $A$ is defined by \eqref{A} with $\mu_0=0$. Then by (A6),  $U(t)\bm x_0$ satisfies
\begin{eqnarray*}
U(t)\bm x_0&=&T_{(A-\gamma Q)_0}(t)\bm x_0+S_{A-\gamma Q}\diamond (\gamma Q+F)(U(\cdot+s)\bm x_0)(t-s)\\
&\ge& T_{(A-\gamma Q)_0}(t)\bm x_0+S_{A-\gamma Q}\diamond (\gamma Q+F_1)(U(\cdot+s)\bm x_0)(t-s), \quad \forall t\ge t',
\end{eqnarray*}
where
\begin{equation*}
F_1\begin{gathered}\begin{pmatrix}
u\\ \begin{pmatrix}
0_Y\\ w
\end{pmatrix}
\end{pmatrix}
=\begin{pmatrix}
\int_{0} ^{a_{\max}}\beta(\cdot, a, \delta_0) w(\cdot, a)da-(m+e+c\delta_0) u\\
\begin{pmatrix}
\chi(\cdot, \delta_0)e u\\
-\mu(\cdot,\cdot, \delta_0) w
\end{pmatrix}
\end{pmatrix}.
\end{gathered}
\end{equation*}
By Proposition \ref{supper solution: integral form}, $U(t)\bm x_0\ge (\check u, (0_Y, \check w))$ for all $t\ge t'$, where $(\check u, (0_Y, \check w))$ is the integrated solution of the following problem
\begin{equation}\label{model_upp}
  \left\{
   \begin{array}{lll}
\partial_{t} u = d\Delta u +\int_0^{a_{\max}}\beta(x, a, \delta_0)w(x, a, t)da-(m+e+c\delta_0)u, \quad &x\in\Omega, t>t', \\
\partial_{t}w+\partial_{a}w= -\mu(x,a, \delta_0)w ,\quad &x\in\Omega, t>t',a\in (0,a_{\max}), \\
\partial_\nu u=0,~\quad &x\in \partial \Omega, t>t', \\
 w(x,0,t)= \chi(x,\delta_0)e(x)u(x,t),  \quad &x\in\Omega, t>t',
   \end{array}
  \right.
  \end{equation}
  with initial data $U(t')\bm x_0$. Define
\begin{equation*}
\tilde{\mathcal{L}}:=
\begin{pmatrix}
d\Delta u+\int_0^{a_{\max}}\beta(\cdot, a, \delta_0) w(\cdot, a)da-(m+e+c\delta_0)u \\
-w(\cdot,0)+\chi(\cdot,\delta_0)e u\\
-w_{a}-\mu(\cdot,\cdot,\delta_0)w
\end{pmatrix},
\quad\forall \bm x=(u, (0_Y, w)) \in D(\tilde{\mathcal{L}}),
\end{equation*}
with $D(\tilde{\mathcal{L}})=D({\mathcal{L}})$. Since $\mathcal{R}_0>1$, by Theorem \ref{theorem_R}, $s(\mathcal{L}_0)>0$ is a principal eigenvalue of $\mathcal{L}_0$. Therefore, we can choose $\delta_0>0$ small such that $s(\tilde{\mathcal{L}}_0)>0$, where $\tilde{\mathcal{L}}_0$ is the part of $\tilde{\mathcal{L}}$ in $\mathbb{X}_0$.  Moreover, using similar arguments as in Theorem \ref{theorem_R}, one can show that $\tilde\lambda_0:=s(\tilde{\mathcal{L}}_0)$ is the principal eigenvalue of $\tilde{\mathcal{L}}_0$ with a positive eigenvector $(\tilde \phi, (0_Y, \tilde \varphi))$.

Firstly, suppose $a_{\max}=\infty$. Similar to \eqref{eigenvalue problem:3}, $\tilde\phi^*:=\tilde\phi$ satisfies
   \begin{equation}\label{phistr1}
  \left\{
   \begin{array}{lll}
\tilde\lambda_0 \tilde\phi^*=d\Delta \tilde\phi^* + \chi(x,\delta_0)eA_{\tilde\lambda_0}-(m+e+c\delta_0)\tilde\phi^*, \quad &x\in\Omega, \\
\partial_\nu \tilde\phi^*=0,~\quad &x\in \partial \Omega,  \\
   \end{array}
  \right.
  \end{equation}
where
$$
A_{\tilde\lambda_0}(x):=\tilde\phi^*\int_0^{a_{\max}} \beta(x, a, \delta_0)  e^{-\int_0^a(\mu(x, l, \delta_0)+\tilde\lambda_0)dl} da, \quad\forall x\in\bar\Omega.
$$
Let
$$
\tilde\varphi^*(x, a):=\tilde\phi^*\int_a^{a_{\max}} \beta(x, s, \delta_0)  e^{-\int_a^s(\mu(x, l, \delta_0)+\tilde\lambda_0)dl} ds, \quad \forall x\in\bar\Omega, a\in [0, a_{\max}).
$$
 It is easy to check that $\tilde\varphi^*$ satisfies
\begin{equation}\label{phistr}
  \left\{
   \begin{array}{lll}
\partial_{a} \tilde\varphi^* = (\mu(x, l, \delta_0)+\tilde\lambda_0)\tilde\varphi^*-\beta(x, a, \delta_0)\tilde\phi^*, \quad &x\in\Omega, a\in (0, a_{\max}), \\
 \tilde\varphi^*(x,0)=A_{\tilde\lambda_0}(x),\quad  &x\in\Omega.
   \end{array}
  \right.
  \end{equation}

Multiplying the first two equations of \eqref{model_upp} by $\tilde \phi^*$ and $\tilde \varphi^*$ and integrating over $\Omega$ and $\Omega\times (0, a_{\max})$, respectively, we obtain
\begin{eqnarray*}
&& \quad \frac{d}{dt}\left( \int_\Omega \check u\tilde\phi^*dx+\int_0^{a_{\max}}\int_\Omega \check w\tilde\varphi^* dxda   \right) \\
&&=\int_\Omega \left[d\Delta\check u+\int_0^{a_{\max}}\beta(x, a, \delta_0)\check w(x, a, t)da-(m+e+c\delta_0)\check u \right] \tilde\phi^* dx\\
&&\quad+\int_0^{a_{\max}}\int_\Omega \left(-\partial_a\check w-\mu(x, a, \delta_0)\check w \right)\tilde\varphi^*dxda\\
&&=\tilde\lambda_0\int_\Omega \check u\tilde\phi^*dx-\int_\Omega \chi(x, \delta_0)e A_{\tilde\lambda_0}\check u\tilde dx+\int_\Omega\left(\int_0^{a_{\max}} \beta(x, a, \delta_0)\check w(x, a, t)da\right) \tilde\phi^* dx\\
&&\quad+\int_0^{a_{\max}}\int_\Omega (\partial_a\tilde\varphi^*-\mu(x, a, \delta_0)\tilde\varphi^*)\check w dx da+\int_\Omega \check w(x, 0, t)\tilde\varphi^*(x, 0)dx\\
&&=\tilde\lambda_0 \left( \int_\Omega \check u\tilde\phi^*dx+\int_0^{a_{\max}}\int_\Omega \check w\tilde\varphi^* dxda   \right),
\end{eqnarray*}
where we have used \eqref{phistr1} and \eqref{phistr} in the above computations. Since $\tilde\lambda_0>0$, we have
$$
\int_\Omega \check u\tilde\phi^*dx+\int_0^{a_{\max}}\int_\Omega \check w\tilde\varphi^* dxda    \to \infty \ \ \text{as}\ t\to\infty.
$$
This contradicts the boundedness of the solution $U(t)\bm x_0$ of \eqref{model}.

Then, suppose $a_{\max}<\infty$. By \eqref{model_upp}, we have
$$
\check w(x, a, t)=\chi(x, \delta_0)e\check u(x, t-a)e^{-\int_0^a\mu(x, l, \delta_0) dl},\quad \forall t>a_{\max}.
$$
Substituting it into the first equation of \eqref{model_upp}, we obtain
\begin{equation}\label{model_upp11}
  \left\{
   \begin{array}{lll}
\partial_{t} \check u = d\Delta \check u +\int_0^{a_{\max}}\beta(x, a, \delta_0)\chi(x, \delta_0)e\check u(x, t-a)e^{-\int_0^a\mu(x, l, \delta_0) dl}da-(m+e+c\delta_0)\check u, \\
\partial_\nu \check u=0,  \\
   \end{array}
  \right.
  \end{equation}
for $t>a_{\max}$. Similar to \eqref{eigenvalue problem:3}, $(\tilde\lambda_0, \tilde\phi^*)$ satisfies
  \begin{equation*}
  \left\{
   \begin{array}{lll}
\lambda \phi=d\Delta \phi + \chi(x,\delta_0)e(x)\displaystyle\int_{0} ^{a_{\max}}\beta(x,a,\delta_0)e^{-\int_{0}^{a}\mu(x,l,\delta_0)dl}e^{-\lambda a}da\phi-(m+e+c\delta_0)\phi, \quad &x\in\Omega, \\
\partial_\nu \phi=0,~\quad &x\in \partial \Omega.  \\
   \end{array}
  \right.
  \end{equation*}
It follows that $\check u(x, t)\ge \epsilon\tilde\phi^* e^{(t-a_{\max})\tilde\lambda_0}$ for all $x\in\bar\Omega$ and $t\ge a_{\max}$, where $\epsilon$ is small such that $\check u(\cdot, a_{\max})\ge \epsilon \tilde\phi^*$. Noticing $\tilde\lambda_0>0$, this contradicts  the boundedness of $u$.

\vspace{0.5cm}
\noindent Finally, combining Steps 1-4, by \cite[Section 1.3.2]{zhao2017dynamical},  \eqref{model} is uniformly persistent with respect to $(\mathbb{X}_0^0, \partial \mathbb{X}_0^0, \rho)$ in the sense that there exists $\epsilon_0>0$ such that $\liminf_{t\to\infty} \rho(U(t)\bm x_0)\ge \epsilon_0$ for any $\bm x_0\in \mathbb{X}_0^0$. Hence, \eqref{persist} holds.
\end{proof}

\subsection{Global stability of the positive equilibrium when $\mathcal{R}_0>1$}

This subsection is devoted to investigating the existence and stability of the positive equilibrium when $\mathcal{R}_{0}>1$. An equilibrium  $E=(u, (0_Y, w))\in D(A)$  is called a positive equilibrium if $(u, (0_Y, w))\ge (\not\equiv) (0_Y, (0_Y, 0_Z))$. Clearly, $E$ satisfies
\begin{equation}\label{steady state equation}
  \left\{
   \begin{array}{lll}
d\Delta u +\int_{0} ^{a_{\max}}\beta(x,a, P)w(x,a)da-(m+e)u- cu^{2}=0, \quad &x\in\Omega, \\
\partial_{a}w=-\mu(x,a, P)w,\quad &x\in\Omega, a\in (0, a_{\max}), \\
\partial_\nu u=0, \quad &x\in \partial \Omega, \\
 w(x,0)=\chi(x, P)e(x)u(x),  \quad &x\in\Omega, \\
   \end{array}
  \right.
  \end{equation}
where $P=\int_0^{a_{\max}}\int_\Omega w(x, a)dxda$.

\begin{lemma}
 Suppose that $(A1)$-$(A3)$ and $(A5)$ hold.   Then,  $E=(u,(0_Y, w))$ is a positive equilibrium of \eqref{model} if and only if $(u,P)$ with $u\in D(A_1)$ and $P\in \mathbb{R}_+$ is a positive solution of
 \begin{equation}\label{steady state equation1}
  \left\{
   \begin{array}{lll}
 d\Delta u + e u \int_{0} ^{a_{\max}}\beta(x,a, P)\chi(x, P) e^{-\int_0^a\mu(x, l, P)dl}da-(m+e)u- cu^{2}=0, \quad &x\in\Omega, \\
\partial_\nu u=0, \quad &x\in \partial \Omega, \\
 P=\int_0^{a_{\max}}\int_\Omega\chi(x, P) e u(x) e^{-\int_0^a\mu(x, l, P)dl}dxda,  \quad & \\
   \end{array}
  \right.
  \end{equation}
  and
\begin{equation}\label{ssw}
w(x, a)=\chi(x, P) e(x) u(x) e^{-\int_0^a\mu(x, l, P)dl}, \quad x\in\Omega, a\in [0, a_{\max}).
  \end{equation}
\end{lemma}
\begin{proof}
    Solving the second and fourth equation of system  \eqref{steady state equation}, we obtain \eqref{ssw}. Substituting it into the first equation of \eqref{steady state equation}, we obtain the first equation \eqref{steady state equation1}. The conclusion follows easily from this observation.
\end{proof}

If $E=(u,(0_Y, w))$ is a positive equilibrium of system \eqref{model}, by the elliptic maximum principle, either $u\equiv 0$ on $\bar\Omega$ or $u(x)>0$ for all $x\in\bar\Omega$. Then by \eqref{ssw}, we must have  $u(x)>0$ for all $x\in\bar\Omega$ and $w(x, a)>0$ for all $x\in\bar\Omega$ and $a\in [0, a_{\max})$.

The following assumption is weaker than (A6):
\begin{itemize}
\item[(A7)] The functions
$$[0, \infty)\ni P\to \int_{0} ^{a_{\max}}\beta(\cdot,a, P)\chi(\cdot, P) e^{-\int_0^a\mu(\cdot, l, P)dl}da$$
and
$$[0, \infty)\ni P\to \int_0^{a_{\max}}\chi(\cdot, P) e^{-\int_0^a\mu(\cdot, l, P)dl}da
$$
are decreasing in $P$.
\end{itemize}


\begin{proposition}\label{existence of positive steady state}
Suppose that $(A1)$-$(A3)$, $(A5)$ and $(A7)$ hold. If $\mathcal{R}_{0}>1$, then system \eqref{model} has a unique positive equilibrium.
\end{proposition}
\begin{proof}
Let $\hat \lambda_0$ be the  principal eigenvalue of \eqref{eigenvalue problem:0}. By Proposition \ref{prop_eig0} and $\mathcal{R}_{0}>1$, we have $\hat \lambda_0>0$. For each $P\ge 0$, let $u_P$ be the unique stable equilibrium of
       \begin{equation}\label{fkpp}
  \left\{
   \begin{array}{lll}
d\Delta u + \chi(x,P)e u\displaystyle\int_{0} ^{a_{\max}}\beta(x,a,P)e^{-\int_{0}^{a}\mu(x,l,P)dl}da-(m+e)u-cu^2=0, \quad &x\in\Omega, \\
\partial_\nu u=0,~\quad &x\in \partial \Omega.  \\
   \end{array}
  \right.
  \end{equation}
Let $\hat\lambda_P$ be the principal eigenvalue of
       \begin{equation}\label{eigenvalue problem:P}
  \left\{
   \begin{array}{lll}
\lambda \phi=d\Delta \phi + \chi(x,P)e(x)\displaystyle\int_{0} ^{a_{\max}}\beta(x,a,P)e^{-\int_{0}^{a}\mu(x,l,P)dl}da\phi-(m+e)\phi, \quad &x\in\Omega, \\
\partial_\nu \phi=0,~\quad &x\in \partial \Omega.  \\
   \end{array}
  \right.
  \end{equation}
It is well-known that $u_P>0$ if and only if $\hat\lambda_P>0$, and  $u_P\equiv 0$ if and only if $\hat\lambda_P\le 0$ \cite{Cantrell2003}. By (A7), $u_P$ is decreasing in $P$. Indeed, if $P<P'$ such that $u_{P'}>0$, it is easy to see that $u_{P'}$ is a lower solution of \eqref{fkpp}. We can choose a large positive constant to be an upper solution of \eqref{fkpp} . By the method of upper/lower solutions and the uniqueness of the positive solution, we have $u_P\ge u_{P'}$.

Define $H:[0, \infty)\to [0, \infty)$ by
$$
H(P)=\int_\Omega\int_0^{a_{\max}}\chi(x, P) e u_P(x) e^{-\int_0^a\mu(x, l, P)dl}dadx.
$$
Since $u_P$ is decreasing in $P$ and by (A7), $F$ is decreasing in $P$. Since $\hat\lambda_0>0$, we know $H(0)>0$. Therefore, the equation $H(P)=P$ has a unique positive solution. It follows that \eqref{model} has a unique positive equilibrium.
\end{proof}

 Next, we present a result on the local stability of the positive equilibrium.

\begin{proposition}\label{local stability of positive steady state}
Suppose that $(A1)$-$(A3)$ and $(A5)$-$(A6)$ hold. If $\mathcal{R}_{0}>1$, then the positive equilibrium $E$ of system \eqref{model} is linearly  stable.
\end{proposition}
\begin{proof}
By Theorem \ref{existence of positive steady state}, \eqref{model} has a unique positive equilibrium $E=(u, (0_Y, w))$ . Let $P=\int_0^{a_{\max}} \int_\Omega w(x, a) dxda$.
Linearizing \eqref{model} at $E$, we obtain the following eigenvalue problem:
\begin{equation}\label{eigLp}
\mathcal{L}'\bm\psi+\mathcal{C}\bm\psi=\lambda\bm\psi, \quad \bm\psi=(\phi, 0_Y, \varphi)\in D(\mathcal{L}'+\mathcal{C})=D(\mathcal{L}),
\end{equation}
where
\begin{equation*}
\begin{gathered}\mathcal{L}'\bm\psi
=\begin{pmatrix}
d\Delta\phi+\int_{0} ^{a_{\max}}\beta(\cdot,a,P)\varphi(\cdot, a) da-(m+e+2cu)\phi \\
-\varphi(\cdot,0)+\chi(\cdot,P)e\phi\\
-\varphi_{a}-\mu(\cdot,\cdot,P)\varphi
\end{pmatrix}
\end{gathered}
\end{equation*}
and
\begin{equation*}
\begin{gathered}\mathcal{C}\bm\psi
=\begin{pmatrix}
     \tilde\varphi \int_{0} ^{a_{\max}}\beta_P(\cdot,a,P)w(\cdot, a) da \\
\chi_P(\cdot, P) e u \tilde\varphi\\
-\mu_P(\cdot, \cdot, P)w\tilde\varphi
\end{pmatrix}
\end{gathered}
\end{equation*}
with $\tilde\varphi=\int_0^{a_{\max}}\int_\Omega \varphi(x, a) dxda$. Clearly, $\mathcal{C}$ is compact.

Similar to $\mathcal{L}$ in Section \ref{threshold value}, $\mathcal{L}'$ is resolvent positive, the part of $\mathcal{L}'$ in $\mathbb{X}_0$,  $\mathcal{L}'_0$, is the generator of a positive strongly continuous semigroup $\{T_{\mathcal{L}'_0}(t)\}_{t\ge 0}$ in $\mathbb{X}_0$, and $\omega_{ess}(\mathcal{L}_0')<0$.  Since $\mathcal{C}$ is compact, $\mathcal{L}'+\mathcal{C}$ is resolvent positive, the part of $\mathcal{L}'+\mathcal{C}$ in $\mathbb{X}_0$,  $(\mathcal{L}'+\mathcal{C})_0$, is the generator of a positive strongly continuous semigroup $\{T_{(\mathcal{L}'+\mathcal{C})_0}(t)\}_{t\ge0}$ in $\mathbb{X}_0$,
and $\omega_{ess}((\mathcal{L}+\mathcal{C})_0')\le \omega_{ess}(\mathcal{L}_0')<0$ \cite[Theorem 1.2]{Ducrot2008}. Moreover, by (A6), we have $\mathcal{C}\bm\psi\le 0_\mathbb{X}$ for any $\bm\psi\in \mathbb{X}_0$. It follows from \cite[Corollary VI.1.11]{Engel2000} that $s((\mathcal{L}'+\mathcal{C})_0)\le s(\mathcal{L}'_0)$.

It suffices to show $s(\mathcal{L}'_0)<0$. Suppose to the contrary that $s(\mathcal{L}'_0)\ge 0$. Since $\omega_{ess}(\mathcal{L}_0')<0$, similar to Theorem \ref{theorem_R}, $\lambda'_0:=s(\mathcal{L}'_0)$ is the principal eigenvalue of $\mathcal{L}'_0$ corresponding with a positive eigenfunction $\bm\psi'=(\phi', (0_Y, \varphi'))$. Moreover, similar to \eqref{eigenvalue problem:3}, $\phi'$ satisfies
  \begin{equation}\label{eigpp}
  \left\{
   \begin{array}{lll}
\lambda_0' \phi'=d\Delta \phi' + \chi(x,P)e\displaystyle\int_{0} ^{a_{\max}}\beta(x,a,P)e^{-\int_{0}^{a}\mu(x,l, P)dl}e^{-\lambda_0' a}da\phi'-(m+e+2cu)\phi', \quad &x\in\Omega, \\
\partial_\nu \phi'=0,~\quad &x\in \partial \Omega.  \\
   \end{array}
  \right.
  \end{equation}
Multiplying the first equation of \eqref{eigpp} by $u$ and the first equation of \eqref{steady state equation1} by $\phi'$, taking the difference, and integrating the resulting equation over $\Omega$, we obtain
$$
0\le \lambda_0'\int_\Omega \phi' udx =\int_\Omega \left(\chi(x,P)e\phi' u\displaystyle\int_{0} ^{a_{\max}}\beta(x,a,P)e^{-\int_{0}^{a}\mu(x,l, P)dl}(e^{-\lambda_0' a}-1)da-cu^2\phi'\right)dx<0.
$$
This is a contradiction. Hence,  $\lambda_0'=s(\mathcal{L}'_0)<0$.
\end{proof}

The following result concerns the global attractivity of the positive equilibrium of \eqref{model}.

\begin{theorem}\label{theorem_global}
   Suppose that $(A1)$-$(A5)$ and $(A7)$ hold, and $\mathcal{R}_0>1$. Let $\bm\psi$ be a positive eigenvector of $\mathcal{L}_0$ corresponding to the principal eigenvalue $s(\mathcal{L}_0)$ (see Proposition \ref{prop_eig0}). Let $\bm x_0=(u_0, (0_Y, w_0))\in\mathbb{X}_0$.
   If $\bm x_0\ge \epsilon \bm\psi$ and $w_0\le Me^{-\underline \mu a}$ for some $\epsilon, M>0$, then the solution of \eqref{model} with initial data $\bm x_0$ satisfies
    \begin{equation}\label{conv}
   \lim_{t\to\infty} u(\cdot, t)=u^*\ \ \text{in}\ Y \ \ \text{and}\ \ \lim_{t\to\infty} w(\cdot, \cdot, t)=w^* \ \text{in}\  Z,
    \end{equation}
    where $E=(u^*, (0_Y, w^*))$ is the unique positive equilibrium of \eqref{model}. In addition, if $a_{\max}<\infty$, then \eqref{conv} holds if $\bm x_0\in \mathbb{X}_0^0$.
\end{theorem}
\begin{proof}
By Proposition
\ref{existence of positive steady state}, \eqref{model} has a unique positive equilibrium $E=(u^*, (0_Y, w^*))$.

Let $\underline{\bm\psi}=\epsilon\bm\psi=\epsilon(\phi, (0_Y, \varphi))\in D(A)\cap \mathbb{X}_{0+}$ for some $\epsilon>0$. We claim that, if $\epsilon$ is sufficiently small, then $\underline{\bm\psi}$ satisfies
\begin{equation}\label{lows}
    A\underline{\bm\psi} + F(\underline{\bm\psi})\ge 0_{\mathbb{X}}.
\end{equation}
Indeed, \eqref{lows} is equivalent to
\begin{equation}\label{lows1}
  \left\{
   \begin{array}{lll}
d\Delta \phi +\int_{0} ^{a_{\max}}\beta(x,a, \epsilon \tilde\varphi)\varphi(x,a)da-(m+e)\phi- c\epsilon \phi^{2}\ge 0, \quad &x\in\Omega, \\
-\partial_{a}\varphi-\mu(x,a, \epsilon\tilde\varphi)\varphi\ge 0,\quad &x\in\Omega, a\in (0, a_{\max}), \\
\partial_\nu \phi\ge 0, \quad &x\in \partial \Omega, \\
 -\varphi(x,0)+\chi(x, \epsilon\tilde\varphi)e(x)\phi\ge 0,  \quad &x\in\Omega, \\
   \end{array}
  \right.
  \end{equation}
  where $\tilde\varphi=\int_0^{a_{\max}}\int_\Omega\varphi(x,a) dx da$.
  By $\mathcal{R}_0>1$ and Proposition \ref{prop_eig0}, $\bm\psi$ is a positive eigenvector of $\mathcal{L}_0$ corresponding with principal eigenvalue $\lambda_0:=s(\mathcal{L}_0)>0$. By \eqref{eigenvalue problem:2}, \eqref{lows1} holds if
 \begin{equation}\label{lows2}
  \left\{
   \begin{array}{lll}
\lambda_0 \phi +\int_{0} ^{a_{\max}}\beta(x,a, \epsilon \tilde\varphi)\varphi(x,a)da-\int_{0} ^{a_{\max}}\beta(x,a, 0)\varphi(x,a)da- c\epsilon \phi^{2}\ge 0,  &x\in\Omega, \\
\lambda_0\varphi+\mu(x,a, 0)\varphi-\mu(x,a, \epsilon\tilde\varphi)\varphi\ge 0, &x\in\Omega, a\in (0, a_{\max}), \\
 -\chi(x, 0)e(x)\phi+\chi(x, \epsilon\tilde\varphi)e(x)\phi\ge 0,   &x\in\Omega.
   \end{array}
  \right.
  \end{equation}
Since $\lambda_0>0$, we can choose $\epsilon>0$ small such that the first two inequalities of \eqref{lows2} hold.
Since $\chi$ is increasing in $P$, the third inequality of \eqref{lows2} holds. This verifies \eqref{lows}. By Theorem \ref{Theorem_c}, the solution  $U(t)\underline{\bm\psi}$ of model \eqref{model} with initial data $\underline{\bm\psi}$ is increasing and converges to a positive equilibrium. Since the positive equilibrium $E$ is unique, we must have $U(t)\underline{\bm\psi}\to E$ in $\mathbb{X}$ as $t\to\infty$.

Let $\bar{\bm\psi}=(M_1, 0_Y, M_2e^{-\underline\mu a})\in D(A)\cap \mathbb{X}_{0+}$ for some $M_1, M_2>0$. We claim that one can choose  $M_1$ and $M_2$  such that $\bar{\bm\psi}$ satisfies
\begin{equation}\label{upps}
    A\bar{\bm\psi} + F(\bar{\bm\psi})\le 0_{\mathbb{X}}.
\end{equation}
Indeed, by (A4), \eqref{upps} holds if
\begin{equation}\label{upps1}
  \left\{
   \begin{array}{lll}
\int_{0} ^{a_{\max}} \beta(x, a, \bar P) M_2 e^{-\underline \mu a}da-(m+e)M_1- cM_1^{2}\le 0, \quad &x\in\Omega, \\
M_2 \underline\mu e^{-\underline\mu a}-\mu(x,a, \bar P)M_2 e^{-\underline\mu a}\le 0,\quad &x\in\Omega, a\in (0, a_{\max}), \\
 -M_2+\chi(x, \bar P)e(x)M_1\le 0,  \quad &x\in\Omega, \\
   \end{array}
  \right.
  \end{equation}
  where $\bar P=M_2 \int_0^{a_{\max}} e^{-\mu a} da$. The second inequality of  \eqref{upps1} holds by the assumptions on $\mu$. Therefore, \eqref{upps1} holds if
\begin{equation}\label{M12}
\frac{M_2\bar\beta }{\underline\mu}    \le  \underline c M_1^2 \quad \text{and}\quad \bar\chi \bar e M_1\le M_2.
\end{equation}
We can choose $M_1, M_2$ large with $M_2\ge M$ such that  \eqref{M12} holds. By Theorem \ref{Theorem_c}, $U(t)\bar{\bm\psi}$ is decreasing and converges  to the unique positive equilibrium $E$ in $\mathbb{X}$  as $t\to\infty$.

By the assumptions on $\bm x_0$ and the choices of $\underline{\bm\psi}$ and $\bar{\bm\psi}$, we have $\underline{\bm\psi}\le \bm x_0\le \bar{\bm\psi}$. By Theorem \ref{Monotone Semiflow}, $U(t)\underline{\bm\psi}\le U(t)\bm x_0\le U(t)\bar{\bm\psi}$ for all $t\ge 0$. Since $U(t)\underline{\bm\psi}\bm \to E$ and $U(t)\bar{\bm\psi} \to E$ as $t\to\infty$, we must have $U(t)\bm x_0\to E$ in $\mathbb{X}$ as $t\to\infty$.

Finally, if $a_{\max}<\infty$ and $\bm x_0\in \mathbb{X}_0^0$, by the proof of Theorem \ref{theorem_persist}, $u(x, t)>0$ for all $x\in\bar\Omega$ and $t> a_{\max}$.  By \eqref{wint}, we have
 \begin{equation}\label{wbd}
 w(x, a, t)=\chi(x, P(t-a))e(x) u(x, t-a) e^{-\int_0^a
 \mu(x, a-s, P(t-s))ds}, \ \ \forall t> a.
 \end{equation}
 Therefore, we can choose $\epsilon, M>0$ such that
 $(u(\cdot, a_{\max}+1), w(\cdot, \cdot, a_{\max}+1))\ge \epsilon (\phi, \varphi)$ and $w(\cdot, a, a_{\max}+1)\le Me^{-\underline \mu a}$ for all $a\in [0, a_{\max}]$. Hence, \eqref{conv} holds.
\end{proof}

\begin{remark}
    In Theorem \ref{theorem_global}, if $a_{\max}=\infty$ and (A6) hold ($\chi$ is independent of $P$ as we assume both (A4) and (A6)), then the assumption $w_0\le Me^{-\underline\mu a}$ can be dropped. To see that, suppose $\bm x_0\ge \epsilon \bm\psi$ for some $\epsilon>0$. By \eqref{wbd}, there is $M>0$ such that
    $\limsup_{t\to\infty} w(x, a, t)\le M e^{-\underline\mu a}$ for all $x\in\bar\Omega$ and $a\ge 0$. Since $U(t)\bm x_0\ge U(t) \underline{\bm \psi}$ for all $t\ge 0$ and  $U(t)\bm x_0$ is increasing in $t$, we have $\liminf_{t\to\infty}U(t)\bm x_0 \ge \epsilon \bm\psi$. Therefore, by Theorem \ref{theorem_global}, the $\omega$-limit set $\omega(\bm x_0)$ is contained in the stable set of $E$. By Proposition \ref{local stability of positive steady state}, $E$ is locally asymptotically stable. Then it follows from \cite[Theorem 4.1]{thieme1992convergence} that \eqref{conv} holds.
\end{remark}

\section{Discussion}
In this paper, we investigated the global dynamics of the parabolic-hyperbolic hybrid system \eqref{model}, which models the spatiotemporal behavior of a population with distinct dispersal and sedentary stages. Due to the non-compactness of solution mapping and the presence of nonlocal terms $B(x,t)$ and $P(t)$ in the system, the global analysis require significant technical effort. Following the non-densely defined operator approach established in \cite{Ruan2018}, we reformulated the model \eqref{model} as an abstract Cauchy problem \eqref{Cauchy problem}. Using  the theory developed in \cite{Hale1988}, we  proved the asymptotic smoothness of solution semiflow.

The dependence of the vital rates ($\beta$, $\mu$, and $\chi$) on the nonlocal term $P(t)$ presents considerable challenges when applying the Lyapunov method to analyze the global dynamics of model \eqref{model}. Therefore, we employ the monotone semiflow theory developed in \cite{Magal2019} to examine the global dynamics of the system. Additionally, we derived a biologically interpretable net reproductive rate, $\mathcal{R}_{0}$, and described its connection to the principal eigenvalue of linearized system \eqref{linearized model} (see Theorem \ref{theorem_R}). Under suitable monotonicity conditions, we obtained threshold dynamics results:  $\mathcal{R}_{0}=1$  serves as the boundary separating population persistence from population extinction.

Several avenues for future work are possible.  (1) In our analysis, the monotonicity of the vital rates  with respect to the total number of stationary individuals, $P$,  plays a crucial role. Relaxing these monotonicity assumptions will introduce   challenging problems that require further investigation.   (2) Model \eqref{model} assumes that the effect of intra-specific competition on the settlement proportion, $\chi$, is spatially homogeneous, i.e., $\chi=\chi(x, P(t))$. Alternatively,  one could consider a spatially heterogeneous  nonlocal term by defining $P(x,t)=\int_{0}^{a_{\max}}w(x,a,t)da$  and letting $\chi=\chi(x, P(x, t))$ in model \eqref{model}.    Notably, this modification would make proving the asymptotic smoothness of the solution semiflow for system \eqref{model} more challenging. (3) Model \eqref{model} employs the Laplacian operator to describe the local random movement of dispersing individuals. However, some species display nonlocal dispersal behaviors, where individuals select their steps randomly according to a specific distribution, influenced by various eological factors \cite{Andreu-Vaillo2000,Levin1985,Lutscher2010}. This observation motivates us to extend model \eqref{model} to a novel nonlocal model by replacing the Laplacian operator with an integral operator, where a dispersal kernel can be introduced to capture the nonlocal dispersal behavior. Exploring this revised model offers an intriguing avenue for future research.  (4) The present study assumes that the population habitat is bounded. A fascinating question in spatial ecology is how and at what rate a population spreads through an unbounded domain if it is able to persist. We plan to address this question by analyzing traveling wave solutions and the spreading speeds of system \eqref{model} in an unbounded spatial domain.

\bibliographystyle{abbrv}
\bibliography{reference}
\end{document}